\documentclass[amssymb,amscd,12pt,20pt]{amsart}
\usepackage{amsmath,amsthm}
\usepackage{amsfonts}
\usepackage{amscd}
\usepackage{amscd}
\usepackage[all]{xy}

 \newtheorem{thm}{Theorem}[subsection]
 \newtheorem{cor}[thm]{Corollary}
 
 \newtheorem{prop}[thm]{Proposition}
 \theoremstyle{definition}
 \newtheorem{defn}[thm]{Definition}
 \theoremstyle{remark}
 \newtheorem{rem}[thm]{Remark}

 \newtheorem{ejem}[thm]{Example}
 \newtheorem{ejems}[thm]{Examples}

\numberwithin{equation}{section}
\newcommand{\OO}{{\mathcal O}}
\newcommand{\M}{{\mathcal M}}
\newcommand{\A}{{\mathcal A}}
\newcommand{\U}{{\mathcal U}}
\newcommand{\PP}{{\mathcal P}}
\newcommand{\Nc}{{\mathcal N}}
\newcommand{\KK}{{\mathbb K}}

\newcommand{\ZZ}{{\mathbb Z}}
\newcommand{\Pj}{{\mathbb P}}

\newcommand{\Lc}{{\mathcal L}}
\newcommand{\E}{{\mathcal E}}
\newcommand{\Q}{{\mathcal{Q}}}
\newcommand{\G}{{\mathbb G}}
\newcommand{\GL}{{\mathbb G}{\mathrm L}}
\newcommand{\bGL}{{\mathbf {GL}}}
\newcommand{\AAA}{{\mathbb A}}
\newcommand{\bA}{{\mathbf {A^{1}}}}
\newcommand{\bAn}{{\mathbf {A^{n}}}}
\newcommand{\bAalg}{{\mathbf {A^{1}_{\text{\rm alg}}}}}
\newcommand{\bAalgn}{{\mathbf {A^{n}_{\text{\rm alg}}}}}
\newcommand{\bAtop}{{\mathbf {A^{1}_{\text{\rm top}}}}}
\newcommand{\bAtopn}{{\mathbf {A^{n}_{\text{\rm top}}}}}
\DeclareMathOperator{\Spec}{{Spec}}
\DeclareMathOperator{\Qcoh}{{Qcoh}}
\DeclareMathOperator{\SSpec}{{\mathbb{S}pec}}
\DeclareMathOperator{\Proj}{{Proj}}
\DeclareMathOperator{\Grass}{{Gr}}
\newcommand{\HHom}{{\mathbb H}{\mathrm {om}}}
\newcommand{\EEpim}{{\mathbb E}{\mathrm {pim}}}
\DeclareMathOperator{\limi}{{lim}}
\newcommand{\ilim}[1]{\,\underset{#1}{\underset{\rightarrow}{\limi}}\,}
\DeclareMathOperator{\Ker}{{Ker}} 
\DeclareMathOperator{\Hom}{{Hom}}
\DeclareMathOperator{\Epim}{{Epim}}
\DeclareMathOperator{\Aut}{{Aut}}
\DeclareMathOperator{\Open}{{Open}}
\DeclareMathOperator{\Cov}{{Cov}}
\DeclareMathOperator{\Maps}{{Maps}}
\DeclareMathOperator{\id}{{id}}

\begin{document}

\title[Universal ringed spaces]
 {Universal ringed spaces}

\author{ J. S\'anchez Gonz\'alez}
\address{Departamento de Matem\'{a}ticas, Universidad de 
Plaza de la Merced 1-4, 37008 Salamanca, Spain}
\email{javier14sg@usal.es}

\author{ F. Sancho de Salas}

\address{Departamento de Matem\'{a}ticas and Instituto Universitario de F\'isica Fundamental y Matem\'aticas (IUFFyM), Universidad de 
Plaza de la Merced 1-4, 37008 Salamanca, Spain}
\email{fsancho@usal.es}


\thanks {The  authors were supported by research project MTM2017-86042-P (MEC)}




\begin{abstract} We construct affine spaces, projective spaces and grassmannians in the ca\-te\-gory of ringed spaces. We show how finite posets and sheaves of rings on them appear in a natural way.
\end{abstract}

\maketitle

\section*{Introduction}

Finite ringed spaces are ringed spaces with a particularly simple behaviour, and they have been proved to be useful in the study of other ringed spaces, specially schemes. The use of sheaves over finite posets to describe certain sheaves on general spaces is not a new idea, see for example \cite{DGM}. Certain finite ringed spaces appear also in the algebraic combinatorics literature (see for example \cite{BBFK,Karu} and the references therein). A more systematic treatment of finite ringed spaces and their applications to schemes or other ringed spaces may be found in  \cite{Sancho,Sancho2,Sancho3,Sancho4}. Other formulations in terms of finite quivers and their representations may be found in \cite{EstradaEnochs}. In this paper, rather than using finite ringed spaces as a tool for the study of other spaces, we show that they play their own role in the category of ringed spaces and provide a framework where distinct algebro-geometric and topological methods  do interplay. Let us be more precise.

Our  main aim  is to construct, in the category of ringed spaces, some standard objects such as affine spaces, projective spaces or grassmannians. By this we mean constructing ringed spaces that have analogous universal properties to those of the affine, projective or grassmannian schemes (or differentiable manifolds or analytic spaces). It turns out that these ringed spaces are finite ringed spaces. Let us  motivate why these objects are of some interest.

{\em Ringed spaces versus locally ringed spaces}. The category of ringed spaces shares certain aspects with that of locally ringed spaces. Both of them contain, in a fully faithful way, the category of rings (or the category of affine schemes), but there are also serious differences between them. From out point of view, the category of ringed spaces has two advantages over the category of locally ringed spaces. The first one is that fiber products are very natural and easy to describe: the underlying topological space is the ordinary fiber product of topological spaces and the sheaf of rings is  the tensor product of their corresponding sheaves of rings. On locally ringed spaces, fiber products are much more difficult (see \cite{Gillam}); even for schemes, their existence   is not immediate. The second and major advantage is that the category of ringed spaces contains, fully faithfully, that of topological spaces, since any topological space may be considered as a ringed space by taking the constant sheaf $\ZZ$ as the sheaf of rings. Many general concepts or constructions on ringed spaces are significant at the topological level. Just to give an example, on any ringed space $(S,\OO_S)$ one can consider the category of quasi-coherent modules; in the topological case (i.e., $\OO_S=\ZZ$) this category is (under mild topological  hypothesis) the category of representations of the fundamental group of $S$. As a general philosophy, ringed spaces are a framework where techniques  from both algebraic geometry (schemes) and topology come to interplay. On the contrary, the category of topological spaces is not a fully faithful subcategory of that of locally ringed spaces. Despite all this, one may argue that the main geometric spaces (as affine or projective spaces) are locally ringed spaces (in their different contexts: algebraic, differentiable, analytic) and that they represent  well described functors (functions, invertible quotients), where morphisms are those of locally ringed spaces. Our aim is to show that there exist natural ringed spaces that deserve the name of affine space or projective space.

In order to illustrate this, let us focus our attention by the moment on the projective $k$-scheme $\Pj^n_{k-\text{sch}}:=\Proj k[x_1,\dots,x_{n+1}]$, where $k$ is a ring. This scheme represents the functor, over the category of $k$-schemes, of invertible quotients of the free module of rank $n+1$; moreover, it is not difficult to prove that it represents this functor over the category of all locally ringed spaces over $k$. Furthermore, this functor still makes sense over the category of all ringed spaces over $k$, hence we may ask if there is  a ringed space $\Pj^n_k$ representing this functor. The answer is no, and this could convince someone of the advantage of the category of locally ringed spaces over that of ringed spaces; but  there is a slight and natural modification of the functor that makes it representable. For this,  let us  first see at a glance  why one cannot expect this functor to be representable on ringed spaces. Assume that $S$ is a ringed space representing the functor of invertible quotients of the free module of rank $n+1$. Then, there would exist the universal quotient $\epsilon\colon \OO_S^{n+1}\to\Lc$, which may be thought as $n+1$ sections $s_i$ of $\Lc$. Each of these sections define an open subset $U_{s_i}$ of $S$ (the open subset where $s_i$ is invertible). Now let $h\colon S'\to S$ be a morphism of ringed spaces, which should be equivalent to the invertible quotient $h^*(\epsilon)\colon \OO_{S'}^{n+1}\to h^*\Lc$ given by the sections $h^*(s_i)$. We have two collections of open subsets on $S'$. On the one hand, the open subsets $U_{h^*(s_i)}$; on the other hand, the open subsets $h^{-1}(U_{s_i})$. If $h$ is a morphism of locally ringed spaces, these two coincide,  i.e., $U_{h^*(s_i)}=h^{-1}(U_{s_i})$; but if $h$ is just a morphism of ringed spaces, one only has the inclusion $h^{-1}(U_{s_i})\subseteq U_{h^*(s_i)}$; this means that the open subsets $h^{-1}(U_{s_i})$ are not determined by the sections $h^*(s_i)$ and thus $h$ cannot be equivalent to $h^*(\epsilon)$. This indicates how to modify the functor in order to expect representability: we have to add a topological datum  ($n+1$ open subsets) to the algebraic datum given by the invertible quotient. To be more precise, one must consider the functor $F$ on ringed spaces defined as follows: $F(S)$ is the set of couples $([\Lc],\U)$ where $[\Lc]$ is an invertible quotient of $\OO_S^{n+1}$, given by $n+1$ sections $s_i$ of $\Lc$, and $\U=\{ U_1,\dots,U_{n+1}\}$ is a covering of $S$ satisfying $U_i\subseteq U_{s_i}$. This functor {\em is representable} by a ringed space $\Pj^n_k$, which is a ¡finite ringed space!, i.e., the underlying topological space is finite (a finite poset). Moreover, one has a natural morphism of ringed spaces $\pi\colon \Pj^n_{k-\text{sch}}\to\Pj^n_k$ and the induced morphism $\pi^*\colon \Qcoh(\Pj^n_k)\to \Qcoh(\Pj^n_{k-\text{sch}})$ between the categories of quasi-coherent modules is an {\em equivalence}. Finally, the non-representable original functor $F'$ (of invertible quotients) is $\pi_0$-representable, in the sense that $F'(S)$ is identified with the set of connected components of $F(S)$ (this set has a natural partial order, hence a topology).
 
An analogous digress can be made in order to construct ringed spaces $\AAA^n_k$, $\AAA^n_k \hskip -1pt {\scriptstyle -\{0\}}$, $\Grass_k(r,n)$ that have analogous behaviour in the category of ringed spaces to that of the affine scheme, the punctured affine scheme or the grassmannian scheme in the category of schemes (or locally ringed spaces). All these ringed spaces have, as underlying topological space, either $\PP_n=$ poset of subsets of a finite set $\Delta_n=\{1,\dots,n\}$, or the poset $\PP_n^*$ of non-empty subsets of $\Delta_n$. Taking into account that $\PP_n$ is the realm of abstract simplicial complexes (in the sense that  finite abstract simplical complexes are   closed subsets of $\PP_n$), this gives a new light of the interplay between simplicial complexes an algebro-geometric spaces. We see in section \ref{Stanley-Reisner} how Stanley-Reisner theory fits naturally in this context.

Let us give now a more detailed description of the different parts of the paper. Section 1 is devoted to give elementary general facts about ringed spaces, locally ringed spaces, distributive lattices and the space of parts. There are no original results. If there is any proof is just for the reader convenience.

Section 2 is devoted to the affine ringed space that we have denoted $\AAA^n_k$. It also could be named the $n$-dimensional Stanley-Reisner ringed space, since the underlying topological space is $\PP_n$, the standard combinatorial $n-1$-simplex, and the sheaf of rings is given by square-free monomial localizations of the polinomial ring $k[x_1,\dots,x_n]$. We also introduce what we have called the algebraic and the topological $n$-dimensional affine ringed spaces, respectively denoted by $\AAA^n_\text{alg}$ and $\AAA^n_\text{top}$. The concepts underlying these spaces are what we have called algebraic, topological and algebro-topological functions. An algebraic function on a ringed space $(S,\OO_S)$ is just an element of $\OO_S(S)$; a topological function is an open subset of $S$ (equivalenty, a continous map $S\to \PP_1$); an algebro-topological function is a couple $(f,U)$ where $f$ is an algebric function, $U$ an open subset and $f_{\vert U}$ is invertible. This elementary notion plays a central role in the paper. Throughout the paper, an algebro-topological concept  means a couple made of the usual algebraic concept plus a topological datum (usually a covering) related to it. The ringed spaces $\AAA^n_\text{top}, \AAA^n_\text{alg}, \AAA^n_k$ represent, respectively, the functor of $n$ topological, algebraic and algebro-topological functions. The comparison between them and with the  $n$-dimensional affine $k$-scheme is made in subsection \ref{An-comparisons}. We also introduce in subsection \ref{Gm} the multiplicative group ringed space $\G_m$ and its actions on ringed spaces. The following subsection \ref{punctured} is devoted to  the punctured affine ringed space $\AAA^n_k \hskip -1pt {\scriptstyle -\{0\}}$. It represents the functor of $n$ non simultaneouly vanishing algebro-topological functions. It is the first example where we meet the problem that we outlined above for the projective scheme: The functor on locally ringed spaces represented by the punctured affine scheme is no longer representable on ringed spaces. In this case, the necessary modification of the functor is just to replace algebraic functions by algebro-topological ones. 

Section 3 is devoted to the projective $n$-dimensional ringed space $\Pj_k^n$. It is constructed as the quotient of the ringed space $\AAA^n_k \hskip -1pt {\scriptstyle -\{0\}}$ by the multiplicative group $\G_m$. We prove in Theorem \ref{Pj^n} that it represents the functor of algebro-topological invertible quotients of the free module of rank $n+1$ (see Definition \ref{algtopInvQuot}).

Section 4 deals with the grassmannian ringed space $\Grass_k(r,n)$ representing the functor of algebro-topological quotients of rank $r$ of the free module of rank $n$ (see Definition \ref{algtopQuot}). For its construction, we first construct the ringed space of algebro-topological epimorphisms $k^n\to k^r$; the grassmannian is defined as the quotient of this ringed space by the linear group $\GL_r$ (this group and its actions is studied in subsection \ref{GL}). We conclude this section by constructing the Pl\"ucker embedding for the grassmannian and showing its compatibility with the Pl\"ucker embedding of the grassmannian scheme.

Section 5 is only devoted to show how Stanley-Reisner theory may be reinterpreted in a natural way from our perspective. We do not give new results; we just mention some natural new problems in Stanley-Reisner theory  arising from this perpective.

Section 6 is intended to show some elementary functors on ringed spaces which are representable by a ringed space  which does not have an analogous locally ringed space; in other words,   the above examples (affine space, projective space, grassmannian) may be viewed as finite ringed spaces associated to the affine, projective or grassmannian schemes with respect to standard coverings on them; in this section we give examples of finite ringed spaces representing natural functors and that are not the finite ringed space associated to any scheme.

\section{Generalities}

\subsection{Representable functors}

In this paper we shall make without mention an extensive use of Yoneda's lemma; thus, we shall frequently define a morphism between ringed spaces by defining the morphism between the functors that they represent. This is Grothendieck's standard technique of the functor of points of  schemes.

Let $\mathcal C$ be a category. Recall that a contravariant functor $F\colon\mathcal C\to \{\text{Sets}\}$ is called {\em representable} if there exists an object $R$ of $\mathcal C$ and an element $\xi\in F(R)$ such that the map
\[ \aligned \Hom_{\mathcal C}(A,R)&\to F(A)\\ f&\mapsto F(f)(\xi)\endaligned \] is bijective for any $A\in\mathcal C$. In this case we say that $F$ is representable by $R$ and that $\xi$ is the universal object on $R$. Yoneda's lemma states that if $F,F'$ are two representable functors (with representants $R$ and $R'$ respectively), then one has a bijection
\[\Hom (F,F')=\Hom_{\mathcal C}(R,R')\] where $\Hom(F,F')$ denotes the set of morphisms of functors bewtween $F$ and $F'$.

\subsection{Ringed spaces over $k$}

Throughout this paper $k$ denotes a commutative ring with unit. The topological space with one element shall be denoted by $*$. 

\begin{defn} A {\em ringed space over $k$} is a ringed space $(S,\OO_S)$ such that $\OO_S$ is a sheaf of $k$-algebras. In other words, it is a ringed space over the ringed space $(*,k)$. A morphism of ringed spaces over $k$ is a morphism of ringed spaces $(f,f_\#)\colon (S',\OO_{S'})\to (S,\OO_S)$ such that $f_\#\colon \OO_S\to f_*\OO_{S'}$ is a morphism of sheaves of $k$-algebras. In other words, it is a morphism of ringed spaces over $(*,k)$. Given two ringed spaces over $k$, $(S',\OO_{S'})$ and  $(S,\OO_S)$, we shall denote by $\Hom_k(S',S)$ the set of morphisms of ringed spaces over $k$ from $(S',\OO_{S'})$ to $(S,\OO_S)$.

Any open subset $U$ of a ringed space $S$ inherits a ringed space structure by taking $\OO_U={\OO_X}_{\vert U}$, and we say that it is an {\em open ringed subspace}.

The {\em fiber product} $S\times_TS'$ of two morphisms of ringed spaces $S\to T$ and $S'\to T$ over $k$ is the ringed space $$S\times_TS':=(S\times_T S',\OO_{S\times_T S'})$$ where $S\times_T S'$ is the topological fiber product  and $\OO_{S\times_T S'}=\pi_S^{-1}\OO_S\otimes_{\pi_T^{-1}\OO_T}\pi_{S'}^{-1}\OO_{S'}$, with $\pi_S\colon S\times_T S'\to S$ (resp. $\pi_{S'}$, $\pi_T$) the natural map. For any ringed space $Z$ over $k$ one has
\[ \Hom_k(Z,S\times_TS')=  \Hom_k(Z,S ) \times_{\Hom_k(Z,T)}  \Hom_k(Z, S').\]
\end{defn}

\begin{ejems}\label{top} $\,$\medskip

(1) Let us consider the forgetful functor
\[ \aligned \{\text{Ringed spaces over }k\} &\to \{\text{Topological spaces}\} \\ (T,k)&\mapsto T\endaligned.\]
This functor has a right adjoint: any topological space $T$ may be viewed as a ringed space over $k$, by taking $\OO_T=k$ the constant sheaf on $T$. We have a fully faithful functor
\[ \aligned \{\text{Topological spaces}\}&\to \{\text{Ringed spaces over }k\}\\ T&\mapsto (T,k)\endaligned\] 
which is a right adjoint of the forgetful functor; thus, for any ringed space $S$ and any topological space $T$, one has
\[ \Hom_k(S, (T,k))=\Hom_{\text{cont}}(S,T).\]
The forgetful functor also has a left adjoint: any topological space $T$ may be viewed as a ringed space over $k$, by taking $\OO_T=0$. We have a fully faithful functor
\[ \aligned \{\text{Topological spaces}\}&\to \{\text{Ringed spaces over }k\}\\ T&\mapsto (T,0)\endaligned\] 
which is a left adjoint of the forgetful functor; thus, for any ringed space $S$ and any topological space $T$, one has
\[ \Hom_k( (T,k), S)=\Hom_{\text{cont}}(T,S).\]

(2) Any $k$-algebra $A$ may be viewed as a ringed space over $k$: the punctual ringed space $(*,A)$. We have a fully faithful functor
\[ \aligned \{ k-\text{algebras} \}&\to \{\text{Ringed spaces over }k\}\\ A&\mapsto (*,A)\endaligned\] which has a left adjoint:   $(S,\OO_S)\mapsto \OO_S(S)$. Thus, for any ringed space $S$ and any $k$-algebra $A$ one has
\[ \Hom_k(S,(*,A))=\Hom_{k\text{-alg}}(A,\OO_S(S)).\]
This says that  the punctual ringed space  $(*,A)$ plays the role in the category of ringed spaces that the affine scheme $\Spec A$ does in the category of locally ringed spaces. Moreover, one has a canonical morphism of ringed spaces $\Spec A\to (*,A)$.\medskip

(3) Ringed posets. Let $(S,\OO_S)$ be a ringed space such that $S$ is a poset, i.e., the topology on $S$ is given by a partial order $\leq $ on $S$ (a subset $U$ of $S$ is open if and only if $s\in U$ and $s'\geq s$ implies $s'\in U$). A sheaf of rings $\OO_S$ on $S$ is equivalent to the following data:

- a ring $\OO_s$ for each $s\in S$,

- a morphism of rings $r_{ss'}\colon \OO_s\to\OO_{s'}$ for any $s\leq s'$,

\noindent such that $r_{ss}=\id_{\OO_s}$ and $r_{s's''}\circ r_{ss'}=r_{ss''}$ for any $s\leq s'\leq s''$.

When $S$ is a finite poset, then we say that $(S,\OO_S)$ is a finite ringed space. These appear in a natural way when considering a finite covering of a ringed space:\medskip

(3.1) {\em Finite ringed space associated to a finite covering of a ringed space} (see \cite{Sancho3}, \cite{Sancho4}). Let $(S,\OO_S)$ be a ringed space and let $\U=\{U_1,\dots,U_n\}$ be a finite open covering of $S$. For each $s\in S$, let $U^s=\underset{ U_i\ni s}\cap U_i$ and let us define the equivalence relation $\sim$ on $S$ by:
\[ s\sim s'\Leftrightarrow U^s=U^{s'}\] Let $X$ be the (finite) quotient set $X:=S/\sim$, and let us consider the topology  on $X$ given by the partial order $[s_1]\leq [s_1]\text { iff } U^{s_1}\supseteq U^{s_2}$; The quotient map $\pi\colon S\to X$ is continuous and we define $\OO_X:=\pi_*\OO_S$; thus $(X,\OO_X)$ is a finite ringed space and $\pi$ is a morphism of ringed spaces. If $\U=\{S\}$, then the associated finite ringed space is $(*,\OO_S(S))$.
\end{ejems}

\begin{defn} A {\em quasi-coherent module}  on a ringed space $(S,\OO_S)$ is an $\OO_S$-module $\M$ which is locally a cokernel of free modules; i.e., there exists an open covering $\{ U_i\}$ of $S$ and, for each $i$, an exact sequence of $\OO_{U_i}$-modules
\[ \underset{j\in J}\oplus\OO_{U_i} \to  \underset{j\in J'}\oplus\OO_{U_i} \to\M_{\vert U_i}\to 0\] for some sets of indexes $J,J'$.
\end{defn}

\begin{defn} \label{locallyfree}A {\em locally free $\OO_S$-module of rank $r$} on a ringed space $S$ is an $\OO_S$-module $\E$ which is locally isomorphic to $\OO_S^r$, i.e., there exists an open covering $\{U_i\}$ of $S$ such that $\E_{\vert U_i}\simeq \OO_{U_i}^r$ (isomorphism of $\OO_{U_i}$-modules). If $r=1$, we say that it is an {\em invertible} $\OO_S$-module.
\end{defn}

\begin{defn} For each $\OO_S$-module $\M$, $\wedge^r\M$ denotes its $r$-th exterior power. If $\E$ is locally free of rank $r$, then $\wedge^r\E$ is an invertible $\OO_S$-module. A morphism $\phi\colon \M\to\Nc$ of $\OO_S$-modules induces a morphism $\wedge^r\phi\colon \wedge^r\M\to\wedge^r\Nc$ between their exterior powers.
\end{defn}

\subsection{Locally ringed spaces}
\begin{defn} A {\em locally ringed space} over $k$ is a ringed space $(X,\OO_X)$ over $k$  such that  $\OO_{X,x}$ is a local ring for any $x\in X$. A morphism of locally ringed spaces over $k$ is a morphism of ringed spaces $f\colon X\to Y$ over $k$ such that $\OO_{Y,f(x)}\to \OO_{X,x}$ is a local morphism between local rings for any $x\in X$. Any $k$-scheme  is a locally ringed space over $k$, and a morphism of $k$-schemes is just a morphism of locally ringed spaces over $k$.
\end{defn}

For any locally ringed spaces $S,S'$ over $k$, we denote by $\Hom_{k-\text{loc}}(S,S')$ the set of morphisms of locally ringed spaces over $k$ between them. If they are schemes, it is also denoted by $\Hom_{k-\text{sch}}(S,S')$.

\subsubsection{Affine schemes} Let $A$ be a $k$-algebra and $\Spec A$ its associated affine scheme. For any scheme $S$ over $k$ one has
\[ \Hom_{k-\text{sch}}(S,\Spec A)=\Hom_{k-\text{alg}}(A,\OO_S(S)).\] In other words, $\Spec A$ represents the functor $S\mapsto \Hom_{k-alg}(A,\OO_S(S))$ over the category of $k$-schemes. In fact, $\Spec A$ represents this functor over the category of all locally ringed spaces over $k$; that is:
\begin{prop}[EGA I, 1.6.3]\label{espectroalgebraico} For any locally ringed space $(S,\OO_S)$ over $k$ one has
\[ \Hom_{k-\text{\rm loc}}(S,\Spec A)=\Hom_{k-\text{\rm alg}}(A,\OO_S(S)).\]
\end{prop}

In particular, the affine $n$-dimensional $k$-scheme $\AAA^n_{k-\text{sch}}:=\Spec k[x_1,\dots ,x_n]$ represents the functor $$S\mapsto \OO_S(S)\times\overset{n}\cdots\times\OO_S(S)$$ over the category of locally ringed spaces.

Analogously, the multiplicative group $k$-scheme $\G_m^{k-\text{sch}}:=\Spec k[t,t^{-1}]$ represents the functor
\[ S\mapsto \OO_S(S)^\times\] over the category of locally ringed spaces, where $\OO_S(S)^\times$ denotes the group of units of the ring $\OO_S(S)$. More generally, the scheme $\AAA^n_{k-\text{sch}}\hskip -1pt {\scriptstyle -\{0\}}$ (the complement of the zero section $0\colon \Spec k\to \AAA^n_{k-\text{sch}}$) represents the functor over locally ringed spaces over $k$:
$$ \aligned   S &\mapsto  \{\text{Epimorphisms of modules }\OO_S^n\to\OO_S\}\\ &=\{(f_1,\dots,f_n)\in \OO_S(S)^n: U_{f_1},\dots, U_{f_n} \text{ is a covering}\}\endaligned $$ where, for any $f\in \OO_S(S)$, $U_f$ denotes the (open) subset of $S$ where $f$ is a unit.

\subsubsection{Projective  and grassmannian schemes} Let us denote $\Pj^n_{k-\text{sch}}:=\Proj k[x_0,\dots,x_n]$ the $n$-dimensional projective scheme over $k$. For any $k$-scheme $S$ one has
\[ \Hom_{k-\text{sch}}(S,\Pj^n_{k-sch})=\{\text{Invertible quotients of }\OO_S^{n+1}\}.\] Moreover, $\Pj^n_{k-\text{sch}}$ represents the functor of invertible quotients of $k^{n+1}$ over the category of all locally ringed spaces over $k$; that is:
\begin{prop}\label{proj-scheme1} For any locally ringed space $(S,\OO_S)$ over $k$ one has
\[ \Hom_{k-\text{\rm loc}}(S,\Pj^n_{k-\text{\rm sch}})=\{\text{\rm Invertible quotients of }\OO_S^{n+1}\}.\]
\end{prop}

Analogous result holds for grassmannians. Let $\Grass_{k-\text{sch}}(r,n)$ the scheme representing (over the category of $k$-schemes) the functor of locally  free of rank $r$ quotients  of $k^n$. Then
\begin{prop}\label{grass-scheme1} For any locally ringed space $(S,\OO_S)$ over $k$ one has
\[ \Hom_{k-\text{\rm loc}}(S,\Grass_{k-\text{\rm sch}}(r,n))=\{\text{\rm Locally free of rank $r$ quotients of }\OO_S^{n}\}.\]
\end{prop}

Propositions \ref{proj-scheme1} and \ref{grass-scheme1}  are not necessary for the rest of the paper and they are probably known. They are more intended as a motivation. We leave the proofs---rather easy anyway---to the reader. 

\subsection{Distributive lattices}\label{distributivelattices}

All the results of this subsection may be found in \cite{TeresaSancho}. They are not necessary for the rest of the paper, but we think that they give a very interesting point of view.

\begin{defn} A {\em distributive lattice} is a set $A$ with two binary composition laws (addition and product) satisfying:

(1) $A$ is a commutative semigroup under both operations (we denote by $0$ and $1$ the neutral element relative to the addition and the product respectively).

(2) The product is distributive over the addition.

(3) One has $a\cdot 0=0$, $a^2=a$, $1+a=a$ for any $a\in A$.

A morphism $f\colon A\to B$ of distributive lattices is a map satisfying $f(0)=0$, $f(1)=1$ and $f(a+a')=f(a)+f(a'), f(a\cdot a')=f(a)\cdot f(a')$ for any $a,a'\in A$. The set of morphisms of distributive lattices shall be denoted by $\Hom_{\text{d.l}}(A,B)$.
\end{defn}

Our main example is the following: for any topological space $S$, the set $\A(S)$ of all closed subsets in $S$, with $+$ and $\cdot$ interpreted as intersection and union respectively, is a distributive lattice. A continuous map $f\colon T\to S$ induces a morphism of distributive lattices $f^*\colon \A(S)\to\A(T), a\mapsto f^*(a):=f^{-1}(a)$. One can also consider the set $\Open(S)$ of all open subsets in $S$, with $+$ and $\cdot$ interpreted as union and intersection respectively; it is also a distributive lattice, which is isomorphic to $\A(S)$ (by sending each open subset to its  complementary closed subset).

On a distributtive lattice $A$, it makes sense to define ideals and prime ideals. We shall denote by $\SSpec A$ the set of prime ideals of $A$. It is a topological space with the Zariski topology. 

\begin{thm} (topological analog of Proposition \ref{espectroalgebraico}) For any topological space $S$ one has \[ \Hom_{\text{\rm cont}}(S,\SSpec A)=\Hom_{\text{\rm d.l.}}(A,\A(S)).\]
\end{thm} This theorem says that the functor $A\mapsto \SSpec A$ is right adjoint of the functor $S\mapsto\A(S)$. In particular, one has a canonical morphism $S\to\SSpec \A(S)$. If $S$ is a finite poset, then $S\to\SSpec \A(S)$ is an isomorphism. Analogously, if $A$ is a finite distributive lattice, the natural morphism $A\to \A(\SSpec A)$ is an isomorphism.

 For any distributive lattices $A,B$, we shall denote $A\otimes B$ its tensor product; one has $$\SSpec(A\otimes B)=\SSpec A\times \SSpec B.$$
 We shall denote by $\KK$ the distributive lattice $\KK=\{0,1\}=\A( * )$; then one has
 \[ \SSpec A= \Hom_{\text{d.l.}}(A,\KK).\]  Let $\KK[x]$ be the ``free'' distributive lattice generated by an element $x$; it is a distributive lattice with 3 elements: $\KK[x]=\{0,x,1\}$. This plays the role of polynomials in one variable for distributive lattices, because
\[ \Hom_{\text{d.l}}(\KK[x], A)=A\] for any distributive lattice $A$. Thus, for any topological space one has
\[ \Hom_{\text{cont}}(S,\SSpec \KK[x])=\A(S).\]
  
Now, let us denote $\KK[x_1,\dots,x_n]$ the ``free'' distributive lattice generated by $n$ elements; in other words
\[ \KK[x_1,\dots,x_n]=\KK[x_1]\otimes\cdots\otimes \KK[x_n].\] Then, for any topological space one has
\[ \Hom_{\text{cont}}(S,\SSpec \KK[x_1,\dots,x_n])=\A(S)\times\overset n\cdots\times \A(S).\]
 
\subsection{The space of parts}

For any set $X$, we shall denote by $\PP(X)$ the set of subsets of $X$ and by $\PP^*(X)$ the set of non-empty subsets of $X$. $\PP(X)$ is a poset  with the partial order given by inclusion:
\[ \delta \leq \delta'  \Leftrightarrow \delta \subseteq \delta' .\] Hence it has a topology (a subset $U$ of $\PP(X)$ is open if $\delta\in U$ and $\delta'\geq \delta$ implies $\delta'\in U$). It is clear that $\PP^*(X)$ is an open subset of $\PP(X)$. For our purposes, the set $X$ will be finite, and then $\PP(X)$ is a finite poset.

\begin{defn} The following notations  will be widely used in the paper: \[ \aligned \Delta_n&=\{ 1,\dots, n\}\\ \PP_n&= \PP(\Delta_n)\\ \PP^*_n&= \PP^*(\Delta_n).
\endaligned\]   
For each $i\in\Delta_n$, let $U_{\{i\}}$ be the smallest open subset of $\PP_n$ containing $\{i\}$; that is, $U_{\{i\}}$ is the set of subsets of $\Delta_n$ containing $i$. It is immediate that $U_{\{1\}},\dots, U_{\{n\}}$ is a covering of $\PP_n^*$.
\end{defn}

In  section \ref{Grassmannian-section} we shall also use the following notations: \[ \aligned \Delta_{r,n}&=\text{ set of subsets of } r \text{ elements of }   \Delta_n\\ \PP_{r,n}&= \PP(\Delta_{r,n})\\ \PP^*_{r,n}&= \PP^*(\Delta_{r,n}).
\endaligned\]   

The universal property of $\PP_n$ in the category of all topological spaces is given by the following theorem; for each topological space $S$, let us denote $\Open (S)$ the set of open subsets of $S$. A continuous map $f\colon S'\to S$ induces a map $f^*\colon\Open(S')\to\Open(S)$ such that $U\mapsto f^{-1}(U)$, hence we have a contravariant functor, $\Open$, from the category of topological spaces to the category of sets.

\begin{thm}\label{P_n} For any topological space $S$ one has a functorial isomorphism
\[ \Hom_{\text{\rm cont}}(S,\PP_n)= \Open(S)\times\overset n\cdots\times\Open(S).\] That is, one has an isomorphism of functors: $\Hom_{\text{\rm cont}}(\quad,\PP_n)=\Open\times\overset n\cdots \times\Open$.
\end{thm}

\begin{proof} For each continuous map $f\colon S\to\PP_n$, we have the $n$ open subsets $$f^{-1}(U_{\{1\}}),\dots ,f^{-1}(U_{\{n\}}).$$ Conversely, given $n$ open subsets $U_1,\dots,U_n$ of $S$, we define the map $f\colon S\to \PP_n$ by $f(s)=\{i\in\Delta_n: s\in U_i\}$. It is easy to see that $f$ is continous and that these assignations are functorial and mutually inverse.
\end{proof}

\begin{rem} (1) The proof of the Theorem \ref{P_n} shows that the open subsets $U_{\{1\}},\dots, U_{\{n\}}$ are the universal object of $\PP_n$.

(2)  For $n=1$, $\Delta_1=*$ and $\PP_1$ is the totally ordered poset of two elements $0<1$, where $0$ represents the empty subset of $*$ and $1$ the total subset. The equality
\[ \Hom_\text{cont}(S,\PP_1)=\Open(S)\] maps $f\in\Hom_\text{cont}(S,\PP_1)$ to $f^{-1}(1)$; conversely an open subset $U$ of $S$ corresponds with the continous map $f_U\colon S\to\PP_1$ defined by 
\[ f_U(s)=\left\{\aligned 1,\text{ if }s\in U \\ 0,\text{ if }s\notin U\endaligned\right.\]

(3) For $n\geq 1$  one has \[\PP_n=\PP_1\times\overset{n}\cdots\times \PP_1.\] 

(4) Let us denote by $\Cov^{n}$ the subfunctor of $\Open\times\overset n\cdots\times\Open$ defined as
 \[ \Cov^{n}(S)=\{ (U_1,\dots,U_n)\in\Open (S)^n:  U_1\cup\dots\cup U_n=S\}.\] Then
 \[ \Hom_\text{cont}(S,\PP_n^*)=\Cov^n(S).\] in other words, $\Cov^n$ is representable by $\PP_n^*$;  the covering $U_{\{1\}},\dots, U_{\{n\}}$ of $\PP_n^*$ is the universal object.
 
(5) Theorem \ref{P_n} can be generalized to any set $\Lambda$ instead of $\Delta_n$ in the following way: The functor
\[ \Open^{\Lambda}(S)=\Maps(\Lambda,\Open (S)),\] is representable by the topological space $\PP(\Lambda)=\text{ parts of }\Lambda$. Analogously,  $\Cov^{\Lambda}$ is representable by $\PP^*(\Lambda)=\text{ non empty parts of }\Lambda$.
\end{rem}
\begin{rem}\label{P_n=Spec} For any topological space $S$, one has an identification $\A(S)=\Open(S)$, by sending each closed subset to its complementary open subset. Hence, one has a natural homeomorphism (see subsection \ref{distributivelattices})
\[ \PP_n=\SSpec \KK[x_1,\dots, x_n]\] and an isomorphism of distributive lattices $\A(\PP_n)=\KK[x_1,\dots, x_n]$.
\end{rem}

\section{The affine space $\AAA^n_k$}

Let $S$ be a ringed space over $k$.

\subsection{Algebraic functions}
\begin{defn} An {\em algebraic function} on $S$ is an element of $\OO_S(S)$.
\end{defn}  A morphism of ringed spaces over $k$ $h\colon S'\to S$ induces a morphism of $k$-algebras $h^*\colon \OO_S(S)\to \OO_{S'}(S')$ between their algebraic functions; thus we have a functor
\[ \aligned \bAalg\colon \{\text{Ringed spaces over }k\} &\to \{ \text{$k$-algebras}\}\\ S&\mapsto \bAalg(S)=\OO_S(S)\endaligned\]

\begin{prop} $\bAalg$  is representable by the punctual ringed space over $k$ 
\[ \AAA^1_{\text{\rm alg}}:=(*,k[x]).\] The universal object is given by the algebraic function $x$.
\end{prop}

\begin{proof} $\Hom_k(S,(*,k[x])=\Hom_{k-\text{alg}}(k[x],\OO_S(S))=\OO_S(S)$.
\end{proof}

The following remark is just to emphasize that $\AAA^1_\text{alg}$ is indeed the algebraic affine line: 

\begin{rem}   For any open subset $U$ of $S$, one has an isomorphism of $k$-algebras
\[ \OO_S(U)=\Hom_k(U,\AAA^1_\text{alg}).\]
\end{rem}

\begin{defn} The $n$-dimensional algebraic affine space over $k$ is
\[ \AAA^n_{\text{\rm alg}}:=\AAA^1_{\text{\rm alg}}\times_k\overset n\cdots\times_k \AAA^1_{\text{\rm alg}}= (*,k[x_1,\dots, x_n ]) \] and it represents the functor $\bAalgn=\bAalg\times\overset n\cdots\times\bAalg$. The universal object is given by the algebraic functions $x_1,\dots,x_n$.
\end{defn}

\begin{rem}\label{Lambda-affine} For any set $\Lambda$, one can define $\AAA^\Lambda_{\text{\rm alg}}$ as the punctual ringed space
\[ \AAA^\Lambda_{\text{\rm alg}}=(*,k[X_\Lambda])\] where $k[X_\Lambda]$ is the polynomial ring over $k$ in as many variables as elements of $\Lambda$. Then,
$$\Hom_k(S,\AAA^{\Lambda}_{\text{\rm alg}})=    \Maps(\Lambda,\OO_S(S)).$$ 
\end{rem}

\begin{defn} Each algebraic function $f\colon S\to \AAA^1_{\text{\rm alg}}$ (i.e., $f\in\OO_S(S)$), defines an open subset
\[ U_f:=\{ s\in S: f_s\text{ is a unit of }\OO_{S,s}\},\] where $f_s$ denotes the image of $f$ in $\OO_{S,s}$, i.e.,  $f_s$ is the germ of $f$ at $s$.
\end{defn}

\subsection{Topological functions}
\begin{defn} Let us denote $\A(S)$ the set of closed subsets of $S$, which is a distributive lattice. A {\em topological function} on $S$ is an element of $\A(S)$, i.e., a closed subset of $S$. A morphism of ringed spaces $h\colon S'\to S$ induces a morphism of distributtive lattices $h^*\colon \A(S)\to \A(S')$; thus we have a functor
\[ \aligned \bAtop\colon \{\text{Ringed spaces over }k\} &\to \{ \text{Distributive lattices}\}\\ S&\mapsto \bAtop(S)=\A(S)\endaligned\]
\end{defn}

\begin{prop} $\bAtop$  is representable by the ringed space over $k$ 
\[ \AAA^1_{\text{\rm top}}:=(\PP_1,k).\] By Remark \ref{P_n=Spec}, one has
  $\AAA^1_{\text{\rm top}}=(\SSpec \KK[x],k).$
\end{prop}

\begin{proof} Since taking complementary yields a functorial isomorphism $\A(S)=\Open(S)$, one concludes that $$\A(S)=\Hom_\text{cont}(S,\PP_1)  =\Hom_k(S,  (\PP_1, k)) $$ by Theorem \ref{P_n} and Remark \ref{top}, (1).
\end{proof}

Recall that $\PP_1$ is a totally ordered poset of two elements, and we write $\PP_1=\{ 0<1\}$.  Thus, a topological function $\kappa\colon S\to \AAA^1_{\text{top}}$ corresponds to the closed subset $C_\kappa=\kappa^{-1}(0)$ of $S$ or to the open subset $U_\kappa=\kappa^{-1}(1)$. The open subset $\PP_1^*$ of $\PP_1$ is the universal object.

\begin{rem} In order to emphasize the analogy with the algebraic case, let us mention that the open subset $U_\kappa$ can also be interpreted as the open subset of points where the germ of $\kappa$ is a unit: for each $s\in S$, let us denote $$\A_s:=\ilim{U} \A(U)=\ilim{U} \Hom_k(U,\AAA^1_{\text{top}})$$ where $U$ runs over the open neighbourhoods of $s$. For any $\kappa\colon S\to \AAA^1_{\text{top}}$, let us denote $\kappa_s$ the image of $\kappa$ in $\A_s$. Then
 \[ U_\kappa=\{ s\in S: \kappa_s \text{ is a unit}\}.\]\end{rem}
 
\begin{defn} We shall define
$$\AAA^n_{\text{top}}:= \AAA^1_{\text{top}}\times_k\overset n\cdots\times_k \AAA^1_{\text{top}}=(\PP_n,k).$$ It represents the functor $ \bAtopn=\bAtop\times\overset n\cdots\times\bAtop$. The universal object is given by the open subsets $U_{\{1\}},\dots, U_{\{n\}}$ of $\PP_n$. By Remark \ref{P_n=Spec}, one has
  $$\AAA^n_{\text{\rm top}}=(\SSpec \KK[x_1,\dots,x_n],k).$$
\end{defn} 
 
\begin{rem} For any set $\Lambda$ one can define $\AAA^\Lambda_{\text{top}}:=(\PP(\Lambda),k)$ and one has
\[ \Hom_k(S, \AAA^\Lambda_{\text{top}})=\operatorname{Maps}(\Lambda,\A(S)).\]
\end{rem}
 
\begin{defn} $\AAA^n_{\text{top}}$ has a unique closed point (the empty subset of $\Delta_n$ or the maximal ideal $(x_1,\dots,x_n)$), whose complement shall be denoted by $\AAA^n_{\text{top}}\hskip -1pt {\scriptstyle -\{0\}} $. Thus
\[ \AAA^n_{\text{top}}\hskip -1pt {\scriptstyle -\{0\}} =(\PP^*_n,k)\] and it represents the functor 
\[\aligned S\mapsto & \{(a_1,\dots,a_n)\in\A (S)^n:a_1+\cdots +a_n=1\}
 \\ =&\{ (U_1,\dots,U_n)\in\Open(S)^n: U_1\cup\dots\cup U_n=S\}.\endaligned \] The universal object is given by the covering $U_{\{1\}},\dots,U_{\{n\}}$ of $\PP^*_n$. 
\end{defn}
 
\subsection{ Algebro-topological functions}
 
 Let $S$ be a ringed space. An algebraic function $f\colon S\to \AAA^1_{\text{alg}}$ defines an open subset $ U_f$, hence  a topological function $\kappa_f\colon S\to \AAA^1_{\text{top}}$. However, the assignation $f\mapsto \kappa_f$ is not functorial: if $h\colon S'\to S$ is a morphism of ringed spaces, then the equality $h^{-1}(U_f)=U_{h^*(f)}$ does not hold (though it does if $h$ is a morphism of locally ringed spaces); only the inclusion  $h^{-1}(U_f)\subseteq U_{h^*(f)}$ is satisfied. Thus, in order to have a functorial behaviour, we shall define

\begin{defn} An {\em algebro-topological function} on $S$ is a couple $\theta=(f,\kappa)$, where $f$ is an algebraic function on $S$ and $\kappa$ is a topological function on $S$, such that $U_f\supseteq U_\kappa$. In other words, it is a couple $(f,U)\in \OO_S(S)\times\Open(S)$ such that $f_{\vert U}$ is a unit. We say that $f$ (resp. $U$) is the {\em algebraic part} (resp. the {\em topological part}) of $\theta$.
\end{defn}

Let us denote $\bA (S) $ the set of algebro-topological functions on $S$. It is a commutative monoid with multiplication
$$(f,\kappa)\cdot(f',\kappa')=(f\cdot f',\kappa\cdot\kappa')$$ and unit $(1,1)$. A morphism $h\colon S'\to S$ of ringed spaces induces a monoid homomorphism $$\aligned \bA (S)&\to\bA (S')\\ (f,\kappa)&\mapsto  (h^*(f),h^*(\kappa))\endaligned$$ because $U_{h^*(f)}\supseteq h^{-1}(U_f)\supseteq h^{-1}(U_\kappa)=U_{h^*(\kappa)}$.

\begin{thm} The functor $\bA$ is representable by a ringed space $$\AAA^1_k:=(\PP_1,\OO_{\AAA^1_k}),$$ where $\PP_1=\{ 0<1\}$ and $\OO_{\AAA^1_k}$ is the sheaf of rings on $\PP_1$  given by the rings
\[ \OO_{\AAA^1_k ,{\scriptscriptstyle 0}} =k[x],\quad \OO_{\AAA^1_k ,{\scriptscriptstyle 1}}=k[x,x^{-1}]\] and the natural inclusion $k[x]\to k[x,x^{-1}]$.
\end{thm}

\begin{proof} A morphism of ringed spaces $S\to (\PP_1,\OO_{\AAA^1_k})$ is a couple $(\kappa,\kappa_\#)$, where $\kappa\colon S\to \PP_1$ is a continuous map (i.e., $\kappa$ is a topological function) and $\kappa_\#\colon\OO_{\AAA^1_k}\to \kappa_*\OO_S$ is a morphism of sheaves of $k$-algebras. Now, $\kappa_\#$ is equivalent to giving morphisms of $k$-algebras
\[ k[x]=\OO_{\AAA^1_k ,{\scriptscriptstyle 0}} \to (\kappa_*\OO_S)_0=\OO_S(S),\qquad k[x,x^{-1}]=\OO_{\AAA^1_k ,{\scriptscriptstyle 1}} \to (\kappa_*\OO_S)_1=\OO_S(U_\kappa)\] such that the diagram
\[ \xymatrix{ k[x]\ar[r]\ar[d] & \OO_S(S)\ar[d] \\ k[x,x^{-1}]\ar[r] & \OO_S(U_{\kappa})}\] is commutative; thus  $\kappa_\#$ is equivalent to giving  an element $f\in \OO_S(S)$ (the image of $x$) such that $f_{\vert U_\kappa}$ is a unit.
\end{proof}

\begin{rem} The universal object of $\AAA^1_k$ is given by the couple $(x,U_{\{ 1\}})$, and  one has $U_x=U_{\{1\}}$. Notice that $x$ is the universal object of $\AAA^1_{\text{alg}}$ and $U_{\{ 1\}}$ the universal object of $\AAA^1_{\text{top}}$.
\end{rem}

\begin{defn} We shall denote $$\AAA^n_k:=\AAA^1_k\times_k\overset n\cdots\times_k\AAA^1_k.$$  It represents the functor $\bAn=\bA\times\overset n\cdots\times\bA$.
By definition, $$\AAA^n_k=(\PP_n,\OO_{\AAA^n_k})$$ where $\OO_{\AAA^n_k}$ is the sheaf of rings on $\PP_n$ defined as follows: for each $\delta\in\PP_n$,
\[\OO_{\AAA^n_k,{\scriptscriptstyle\delta}}=k[x_1,\dots,x_n]_{x_{\delta}}\] where $x_{\delta}=\underset{i\in\delta}\prod x_i$, and for any $\delta\leq  \delta'$ the morphism $\OO_{\AAA^n_k,{\scriptscriptstyle\delta}}\to \OO_{\AAA^n_k,{\scriptscriptstyle \delta'}}$ is the  natural morphism $k[x_1,\dots,x_n]_{x_{\delta}}\to k[x_1,\dots,x_n]_{x_{\delta'}}$, $x_j\mapsto x_j$ for any $j$.
\end{defn}

\begin{rem} Since $\delta=\emptyset$ is a minimum of $\PP_n$, one has
\[ \OO_{\AAA^n_k}(\AAA^n_k)= \OO_{\AAA^n_k,{\scriptscriptstyle\emptyset}}=k[x_1,\dots,x_n].\] The universal object of $\AAA^n_k$ is given by the algebro-topological functions $$(x_1,U_{\{1\}}), \dots, (x_n,U_{\{n\}})$$ and one has $U_{x_i}=U_{\{i\}}$.  
\end{rem}

\begin{rem} For any set $\Lambda$ one can define $\AAA^\Lambda_k=(\PP(\Lambda ),\OO_{\AAA^\Lambda_k})$, where $\OO_{\AAA^\Lambda_k}$ is defined by
\[ \OO_{\AAA^\Lambda_k,{\scriptscriptstyle\delta}}=k[x_{\Lambda}]_{x_{\delta}} \] for each subset $\delta$ of $\Lambda$. Thus, one has  \[\Hom_k(S,\AAA^\Lambda_k)= \operatorname{Maps}(\Lambda, \bA(S)).\] 
\end{rem}

\subsection{Comparison between $\AAA^n_k$, $\AAA^n_{\text{\rm alg}}$, $\AAA^n_{\text{\rm top}}$ and $\AAA^n_{k-\text{\rm sch}}$}\label{An-comparisons}$\,$\medskip

The algebraic functions $x_1,\dots,x_n$ on $\AAA^n_k$ define a morphism $$\pi_\text{alg}\colon \AAA^n_k\to\AAA^n_{\text{alg}}$$ whose functorial description is $(\theta_1,\dots,\theta_n)\mapsto (f_1,\dots, f_n)$, with $f_i=$ algebraic part of $\theta_i$.  This morphism is  an homotopy equivalence (in the sense of \cite{Sancho3}). In particular,
\[\Qcoh(\AAA^n_{\text{\rm alg}})\simeq \Qcoh(\AAA^n_k).\]
 Moreover, let us see that $\AAA^n_{\text{alg}}$ may be viewed as the space of ``connected components'' of $\AAA^n_k$. For any ringed space $S$ over $k$, the set $\Hom_k(S,\AAA^n_k)$ has a natural partial order (essentially induced by the partial order of the underlying topological space $\PP_n$ of $\AAA^n_k$), hence a topology. The partial order is:
\[ (\theta_1,\dots,\theta_n)\leq (\theta'_1,\dots,\theta'_n) \Leftrightarrow f_i=f'_i\text{ and } U_i\subseteq U'_i\text{ for every } i=1,\dots,n,\] 
where $\theta_i=(f_i,U_i), \theta'_i=(f'_i,U'_i)$. Let us consider the map 
$$\aligned i(S)\colon \Hom_k(S,\AAA^n_{\text{alg}})&\to \Hom_k(S,\AAA^n_k)\\ (f_1,\dots,f_n)&\mapsto (\theta_1,\dots,\theta_n),\text{ with } \theta_i=(f_i,U_{f_i}).\endaligned$$ 
The composition $\pi_\text{alg}(S)\circ i(S)$ is the identity and $i(S)\circ \pi_\text{alg}(S)\leq \id$.   This means that $\Hom_k(S,\AAA^n_{\text{alg}})$, with the discrete topology, is a deformation retract of $\Hom_k(S,\AAA^n_k)$. The elements in the image of $i(S)$  are precisely the maximal elements of $\Hom_k(S,\AAA^n_k)$ (for the partial order) and each $(\theta_1,\dots,\theta_n)$ is dominated by a unique maximal element. Thus  $\Hom_k(S,\AAA^n_{\text{alg}})$ is identified   with both the set of irreducible components  and   the set of connected components of $\Hom_k(S,\AAA^n_k)$. In conclusion:
\[ \Hom_k(S,\AAA^n_{\text{alg}})=\pi_0\Hom_k(S,\AAA^n_k).\]

The open subsets $U_{\{1\}},\dots,U_{\{n\}}$ of $\AAA^n_k$ define a morphism 
$$\pi_\text{top}\colon \AAA^n_k\to\AAA^n_{\text{top}}$$   functorially described as $(\theta_1,\dots,\theta_n)\mapsto (U_1,\dots, U_n)$, with $U_i=$ topological part of $\theta_i$. This morphism is the identity at the topological level.

Now, let us consider the $k$-scheme $\AAA^n_{k-\text{sch}}:=\Spec[t_1,\dots,t_n]$. The algebro-topological functions $(t_i,U_{t_i})$, $i=1,\dots,n$, of $\AAA^n_{k-\text{sch}}$ define a morphism of ringed spaces over $k$
$$\pi\colon \AAA^n_{k-\text{sch}}\to\AAA^n_k.$$
For any locally ringed space $S$ over $k$, the induced  morphism (composition with $\pi$) $$\Hom_{k-\text{loc}}(S,\AAA^n_{k-\text{sch}})\to \Hom_k(S,\AAA^n_k)$$ is just the map $(f_1,\dots,f_n)\mapsto   (\theta_1,\dots,\theta_n)$, with $\theta_i=(f_i,U_{f_i})$. The morphism $\pi$ has the following universal property: For any locally ringed space $S$ and any morphism of ringed spaces $f\colon S\to \AAA^n_k$ there exists a unique morphism of locally ringed spaces $f_\text{loc}\colon S\to \AAA^n_{k-\text{sch}}$ such that $f\leq \pi\circ f_\text{loc}$.
 
Finally, $\AAA^n_k$ is the finite ringed space associated to the scheme $\AAA^n_{k-\text{sch}}$ and the open covering $\{ \AAA^n_{k-\text{sch}},U_{x_1},\dots,U_{x_n}\}$.  In particular (see \cite{Sancho3}), one obtains an equivalence
\[ \Qcoh(\AAA^n_k)\simeq \Qcoh(\AAA^n_{k-\text{sch}}).\]

\subsection{The multiplicative group}\label{Gm}

\begin{defn} The {\em multiplicative group} over $k$ is the ringed space over $k$:
\[ \G_m^k:=(*,k[x,x^{-1}]).\] The name is justified by the universal property:
\[ \Hom_k(S,\G_m^k)=\OO_S(S)^\times \] for any ringed space $S$ over $k$, where $\OO_S(S)^\times$ denotes the set of units of $\OO_S(S)$. For simplicity, we shall denote $\G_m=\G_m^k$, since $k$ is usually understood.
\end{defn}

\begin{rem} Since $\G_m$ represents the units  of the set $\OO_S(S)$ of algebraic functions, it would have been more natural, as before, to denote it by $\G_m^\text{alg}$ and name it the {\em algebraic} multiplicative group. Let us do so for the moment. One can also ask for the representability of the the functor  $S\mapsto \A(S)^\times$ of topological units and denote its representant by $\G_m^\text{top}$. In fact, since a distributive lattice has only one unit, one has $\A(S)^\times=\{1\}$, and this functor is representable by $\G_m^\text{top} :=(*,k)$. Finally, one might consider the functor $S\mapsto \bA(S)^\times$ of algebro-topological units, ask for its representability and denote by $\G_m$ its representant. But one has an isomorphism
\[ \aligned \OO_S(S)^\times &\overset\sim\to \bA(S)^\times\\ f&\mapsto (f,1)\endaligned \] and hence $\G_m=\G_m^\text{alg}\times_k\G_m^\text{top} = \G_m^\text{alg}$.
\end{rem}

\begin{rem} The functor $S\mapsto \OO_S(S)^\times$ is a functor in (commutative) groups. Thus, $\G_m$ has a (commutative) group structure: one has a multiplication law
\[ \nu\colon\G_m\times_k\G_m\to \G_m,\]   a rational point $(*,k)\to \G_m$,  and an inversion morphism $\G_m\to \G_m$, satisfying the (commutative) group axioms.
\end{rem}

It is well known that giving a $\ZZ$-gradation on a $k$-algebra $A$ is equivalent to giving an action of the multiplicative group scheme $\Spec k[t,t^{-1}]$ on the affine $k$-scheme $\Spec A$. Let us see that this also happens for ringed spaces.

\begin{defn} Let $S$ be a ringed space over $k$. An {\em action} of the multiplicative group $\G_m$ on $S$ is a morphism
\[ \mu\colon \G_m\times_k S\to S\] of ringed spaces over $k$ satisfying the standard axioms of an action; in other words, for any ringed space $T$ over $k$, the map $$\mu(T)\colon \OO_T(T)^\times\times \Hom_k(T,S)\to \Hom_k(T,S)$$ is an action of the group $\OO_T(T)^\times$ on the set $\Hom_k(T,S)$.
\end{defn}

\begin{rem}\label{action} Since $\G_m=(*,k[t,t^{-1}])$, one has $\G_m\times_k S=(S,\OO_S[t,t^{-1}])$ and the action morphism $\mu\colon (S,\OO_S[t,t^{-1}])\to (S,\OO_S)$ is the identity at the topological level (since it must be the identity when taking $t=1$). Thus, the morphism $\mu$  is equivalent to  a morphism of sheaves of $k$-algebras $\mu_\#\colon \OO_S\to \OO_S[t,t^{-1}]$ and the action axioms can be translated into the following conditions: 

(1) (neutral element acts by the identity): The composition $\OO_S\overset{\mu_\#}\to \OO_S[t,t^{-1}]\overset{t\mapsto 1}\to \OO_S$ is the identity. 

(2) (associativity): the diagram 
\[\xymatrix{ \OO_S\ar[r]^{\mu_\#}\ar[d]^{\mu_\#}   & \OO_S[x,x^{-1}]\ar[d]^{\mu_\#\otimes \id} \\ \OO_S[t,t^{-1}]\ar[r]^{\nu_\#\qquad\quad} &  \OO_S[y,y^{-1}]\otimes_{\OO_S}\OO_S[x,x^{-1}]}\] is commutative, where $(\mu_\#\otimes \id)(f x^i)=\mu_\#(f)\otimes x^i$ and $\nu_\#$ is the multiplication of $\G_m$, i.e., $\nu_\#(t)=y\cdot x$.
\end{rem}

\begin{defn} Let $A$ be a $k$-algebra. A $\ZZ$-graded structure on $A$ is a decomposition $A=\underset{i\in\ZZ}\oplus A_i$ of $k$-modules, such that $k\to A$ maps into $A_0$ and $a_i\cdot a_j\in A_{i+j}$ for any $a_i\in A_i,a_j\in A_j$. Let $S$ be a ringed space over $k$. A $\ZZ$-graded  structure on $\OO_S$ means  a $\ZZ$-graded structure on each $k$-algebra $\OO_S(U)$ ($U$ open subset of $S$) such that the restriction morphisms $\OO_S(U)\to\OO_S(V)$ preserve the gradation, for any $V\subseteq U$. Thus $\OO_S$ becomes a sheaf of $\ZZ$-graded $k$-algebras.
\end{defn}

The following proposition is now clear from Remark \ref{action}:
\begin{prop} An action of $\G_m$ on $S$ is equivalent to  a $\ZZ$-graded structure on $\OO_S$.
\end{prop}

\begin{defn}\label{invariants} Let us consider an action of $\G_m$ on $S$. The {\em sheaf of invariants} $\OO_S^{\G_m}$ is defined by:
\[ \OO_S^{\G_m}:=\Ker (\mu_\# - i)\] where $i\colon\OO_S\to \OO_S[t,t^{-1}]$ is the natural inclusion. In other words, for any open subset $U$ of $S$
\[ \OO_S^{\G_m} (U)= [\OO_S(U)]_0\] where $[\quad]_0$ means the subalgebra of degree zero elements.  We shall define the {\em quotient} $S/\G_m$ as the ringed space
\[  S/\G_m:=(S,\OO_S^{\G_m})\] and one has a natural morphism $S\to S/\G_m$ which is $\G_m$-equivariant (considering the trivial action of $\G_m$ on $S/\G_m$).
\end{defn}

\begin{rem} Let $\pi \colon \G_m\times_kS\to S$ be the natural projection (this is the trivial action of $\G_m$ on $S$). Notice that $\pi_\#\colon\OO_S\to \pi_*(\OO_{\G_m\times_kS})=\OO_S[t,t^{-1}]$ is the natural inclusion (denoted by $i$ in Definition \ref{invariants}). Then one has an exact sequence of ringed spaces over $k$ 
\[ \G_m\times_kS \,\aligned \overset\mu\longrightarrow\\[-10pt]\underset\pi\longrightarrow\endaligned \, S\to S/\G_m.
\]
\end{rem}

\begin{ejem} \label{actionGm} $\G_m$ acts on $\AAA^1_k=(\PP_1,\OO_{\AAA^1_k})$ by homotheties: For each ringed space $T$ over $k$ one has
\[ \aligned \OO_T(T)^\times \times \bA(T)&\to \bA(T)\\ (\lambda, (f,\kappa))&\mapsto \lambda\cdot (f,\kappa):=(\lambda f,\kappa)\endaligned \] which is clearly functorial on $T$. 

In the same way, $\G_m$ acts by homotheties on $\AAA^n_k$: since $\AAA^n_k=\AAA^1_k\times_k\cdots\times_k\AAA^1$, the action on $\AAA^n_k$ is induced by the above action on each  factor. Thus, for each $T$ the action is given by the map
\[ \aligned \OO_T(T)^\times \times \bAn(T) &\to \bAn(T) \\ (\lambda, (\theta_1,\dots,\theta_n))&\mapsto \lambda\cdot (\theta_1,\dots,\theta_n):=(\lambda \cdot \theta_1,\dots,\lambda\cdot\theta_n).\endaligned \]
\end{ejem}

\subsection{The punctured affine space $\AAA^n_k \hskip -1pt {\scriptstyle -\{0\}} $}\label{punctured} $\,$

\begin{defn}  We shall denote by $\AAA^n_k\hskip -1pt {\scriptstyle -\{0\}}$ the open ringed subspace of $\AAA^n_k$ whose underlying topological space is $\PP^*_n$. That is
\[ \AAA^n_k\hskip -1pt {\scriptstyle -\{0\}}=(\PP^*_n,\OO_{\AAA^n_k\hskip -1pt {\scriptstyle -\{0\}}}) \] where
$\OO_{\AAA^n_k\hskip -1pt {\scriptstyle -\{0\}}}$ is the restriction of $\OO_{\AAA^n_k}$ to $\PP^*_n$.
\end{defn}

\begin{rem} One has $\G_m=\AAA^1_k\hskip -1pt {\scriptstyle -\{0\}}$ (as one would desire).
\end{rem}

\begin{rem} By definition, $ \AAA^n_k\hskip -1pt {\scriptstyle -\{0\}}=\AAA^n_k\times_{\AAA^n_\text{top}}  \AAA^n_\text{top}\hskip -1pt {\scriptstyle -\{0\}}$. Hence $ \AAA^n_k\hskip -1pt {\scriptstyle -\{0\}}$ represents the functor of $n$ algebro-topological functions whose topological part is a covering; i.e., one has
\[\Hom_k(S, \AAA^n_k\hskip -1pt {\scriptstyle -\{0\}})= \{ (\theta_1,\dots,\theta_n)\in\bAn(S): U_{\theta_1},\dots, U_{\theta_n}\text{ is a covering of }S\}.\]
The universal object is  the restriction of the universal object of $\AAA^n_k$ to $ \AAA^n_k\hskip -1pt {\scriptstyle -\{0\}}$.
\end{rem}

\begin{rem}\label{re-punctured}[Non-representability]\label{norepresentable}
On locally ringed spaces, the functor $$F(S)=\{ (f_1,\dots,f_n)\in\OO_S(S)^n: U_{f_1},\dots, U_{f_n}\text{ is a covering of }S\}$$ coincides with the functor $$F'(S)=\{\text{Epimorphisms  }\OO_S^{n }\to\OO_S\}$$ and they are both representable by the punctured scheme $\AAA^n_{k-\text{sch}}\hskip -1pt {\scriptstyle -\{0\}}$ which is the complement of the zero section $\Spec k\to \AAA^n_{k-\text{sch}}$.   On ringed spaces, these functors do not coincide ($F$ is only a subfunctor of $F'$) and they are not longer representable (for $n>1$, since for $n=1$ both are representable by $\G_m$). Indeed, assume that $F'$ is representable by a ringed space $(X,\OO_X)$. For any topological space $T$ one has
\[ \Hom_{\text{cont}}(T,X)=\Hom_k((T,0),(X,\OO_X))=F'((T,0))= \{\text{Epimorphisms }0^{n }\to 0\}=*.\] Hence $X=*$, i.e., the representant of $F'$ should be a punctual ringed space $(*,A)$. Then, for any locally ringed space $S$
\[ \Hom_{k-\text{loc}}(S,\AAA^{n }_{k-\text{sch}}\hskip -1pt {\scriptstyle -\{0\}})= F'(S)=\Hom_k(S,(*,A))= \Hom_{k-\text{loc}}(S,\Spec A)\] and then $\AAA^{n }_{k-\text{sch}}\hskip -1pt {\scriptstyle -\{0\}}=\Spec A$, which is absurd because $\AAA^{n }_{k-\text{sch}}\hskip -1pt {\scriptstyle -\{0\}}$ is not an affine scheme for $n>1$. The proof for $F$ is completely analogous.
\end{rem}

\begin{rem}[$\pi_0$-representability]   The set $\Hom_k(S,\AAA^n_k\hskip -1pt {\scriptstyle -\{0\}})$ has a natural partial order, which is the restriction of the partial order given in $\Hom_k(S,\AAA^n_k)$ (see  \ref{An-comparisons}). In the same way as in \ref{An-comparisons},  one has a natural morphism $$\pi\colon \Hom_k(\quad,\AAA^n_k\hskip -1pt {\scriptstyle -\{0\}}) \to F$$ and an inclusion $$i(S)\colon F(S)\to 
 \Hom_k(S,\AAA^n_k\hskip -1pt {\scriptstyle -\{0\}})$$  such that the composition $\pi(S)\circ i(S)$ is the identity and  $i(S)\circ \pi(S)\leq \id$.  Again,   $F(S)$ is identified with both the set of irreducible components  and   the set of connected components of $\Hom_k(S,\AAA^n_k\hskip -1pt {\scriptstyle -\{0\}})$. In conclusion, $F$ is $\pi_0$-representable by $\AAA^n_k\hskip -1pt {\scriptstyle -\{0\}}$:
\[ F(S)=\pi_0\Hom_k(S,\AAA^n_k\hskip -1pt {\scriptstyle -\{0\}}).\]
\end{rem}

\begin{rem}[Comparison with the punctured scheme $\AAA^n_{k-\text{sch}}\hskip -1pt {\scriptstyle -\{0\}}$]  
  
The algebro-topological functions $(x_i,U_{x_i})$, $i=1,\dots,n$, of the scheme $\AAA^n_{k-\text{sch}}\hskip -1pt {\scriptstyle -\{0\}}$ define a morphism 
\[ \pi\colon \AAA^n_{k-\text{sch}}\hskip -1pt {\scriptstyle -\{0\}}\to \AAA^n_k\hskip -1pt {\scriptstyle -\{0\}}.\]
For any locally ringed space over $S$ $k$, the induced  morphism (composition with $\pi$) $$\Hom_{k-\text{loc}}(S,\AAA^n_{k-\text{sch}}\hskip -1pt {\scriptstyle -\{0\}})\to \Hom_k(S,\AAA^n_k\hskip -1pt {\scriptstyle -\{0\}})$$ is just the map $(f_1,\dots,f_n)\mapsto   (\theta_1,\dots,\theta_n)$, with $\theta_i=(f_i,U_{f_i})$. For each morphism of ringed spaces $f\colon S\to \AAA^n_k\hskip -1pt {\scriptstyle -\{0\}}$ there exists a unique morphism of locally ringed spaces $f_\text{loc}\colon S\to \AAA^n_{k-\text{sch}}\hskip -1pt {\scriptstyle -\{0\}}$ such that $f\leq \pi\circ f_\text{loc}$.

Finally, $ \AAA^n_k\hskip -1pt {\scriptstyle -\{0\}}$ is the finite ringed space associated to $ \AAA^n_{k-\text{sch}}\hskip -1pt {\scriptstyle -\{0\}}$ and the affine open covering $U_{x_1},\dots,U_{x_n}$. Hence one has an equivalence
\[ \Qcoh(\AAA^n_{k-\text{sch}}\hskip -1pt {\scriptstyle -\{0\}})\simeq \Qcoh(\AAA^n_{k}\hskip -1pt {\scriptstyle -\{0\}})\]
\end{rem}

\section{The projective space $\Pj^n_k$.}

\begin{defn}  Let us consider the action of $\G_m$ on $\AAA^{n+1}_k$ by homotheties (see Example \ref{actionGm}). It is immediate to see that the open ringed subspace $\AAA^n_k\hskip -1pt {\scriptstyle -\{0\}}$ is invariant by this action (notice that $U_{\lambda\theta}=U_\theta$ for any $\lambda\in\OO_T(T)^\times$, $\theta\in\bA(T)$). We shall define the {\em $n$-dimensional projective space over $k$} as:
\[ \Pj^n_k:=(\AAA^{n+1}_k\hskip -1pt {\scriptstyle -\{0\}})/\G_m.\]
By definition, $$\Pj^n_k=(\PP^*_{n+1},\OO_{\Pj^n_k} )$$ with $\OO_{\Pj^n_k}=(\OO_{\AAA^{n+1}_k\hskip -1pt {\scriptscriptstyle -\{0\}}})^{\G_m}$; thus, for each $\delta\in\PP^*_{n+1}$
\[ \OO_{\Pj^n_k,{\scriptscriptstyle\delta}} = (\OO_{\AAA^{n+1}_k\hskip -1pt {\scriptscriptstyle -\{0\}},{\scriptscriptstyle\delta} })^{\G_m}= [k[x_1,\dots,x_{n+1}]_{x_\delta}]_0,\] and for each $\delta\leq\delta'$ the morphism $ \OO_{\Pj^n_k,{\scriptscriptstyle\delta}}\to  \OO_{\Pj^n_k,{\scriptscriptstyle\delta'}}$ is obtained by taking $[\quad]_0$ in the morphism $k[x_1,\dots,x_{n+1}]_{x_\delta}\to k[x_1,\dots,x_{n+1}]_{x_{\delta'}}$.
\end{defn}

\subsubsection{The canonical invertible module $\OO_{\Pj^n_k}(1)$}  

\begin{defn} The canonical $\OO_{\Pj^n_k}$-module $\OO_{\Pj^n_k}(1)$ is defined as the 1-degree elements of $\OO_{\AAA^{n+1}_k\hskip -1pt {\scriptscriptstyle -\{0\}}}$; that is, for each $\delta\in\PP^*_{n+1}$:
\[ [\OO_{\Pj^n_k}(1)]_{\delta}:=[\OO_{\AAA^{n+1}_k\hskip -1pt {\scriptscriptstyle -\{0\},\delta}}]_1=[k[x_1,\dots,x_n]_{x_{\delta}}]_1\] and for each $\delta\leq\delta'$ the morphism $[\OO_{\Pj^n_k}(1)]_{\delta}\to [\OO_{\Pj^n_k}(1)]_{\delta'}$ is obtained by taking  $[\quad]_1$ in the morphism $\OO_{\AAA^{n+1}_k\hskip -1pt {\scriptscriptstyle -\{0\},\delta}} \to \OO_{\AAA^{n+1}_k\hskip -1pt {\scriptscriptstyle -\{0\},\delta'}}$. It is clear that $\OO_{\Pj^n_k}(1)$ is an $\OO_{\Pj^n_k}$-module.
\end{defn}

\begin{rem}\label{O(1)} Each variable $x_i$ defines by multiplication a morphism $\OO_{\Pj^n_k}  \overset{\cdot x_i}\to  \OO_{\Pj^n_k}(1)$, i.e., a global section of $\OO_{\Pj^n_k}(1)$. This morphism is an isomorphism when restricting to $U_{\{i\}}$, hence  $\OO_{\Pj^n_k}(1)$ is an invertible module (Definition \ref{locallyfree}). Moreover,  $U_{x_i}=U_{\{i\}}$. Thus, the triplet
\[ (\OO_{\Pj^n_k}(1),(x_1,\dots,x_{n+1}), \{ U_{\{1\}},\dots,U_{\{n+1\}}\})\] defines what we shall call an algebro-topological invertible quotient of $\OO_{\Pj^n_k}^{n+1}$. 
\end{rem}

\subsubsection{Invertible quotients and strict invertible quotients}\label{invertiblequotients}

\begin{defn} Let $e\in\Lc(S)$ be a global section of an invertible $\OO_S$-module $\Lc$; it  may be thought as a morphism of $\OO_S$-modules $e\colon\OO_S\to\Lc$. We define
\[ U_e:=\{ s\in S: e_s\colon\OO_{S,s}\to\Lc_s\text{ is an isomorphism}\}.\] Is is immediate to see that $U_e$ is an open subset.  More generally, giving a morphism of $\OO_S$-modules $\OO_S^{n+1}\to\Lc$ is equivalent to giving $n+1$ global sections  ${\underbar e}_{n+1} =(e_1,\dots,e_{n+1})$ of $\Lc$, and we shall denote  
$$\U_{{\underbar e}_{n+1}}:=\{ U_{e_1} ,\dots,U_{e_{n+1}}\}.$$
\end{defn}

\begin{defn}\label{epi-strict} We say that ${\underbar e}_{n+1}$ is an {\em epimorphism} if  for any $s\in S$, the morphism $\OO_{S,s}^{n+1}\to\Lc_s$ is an epimorphism. We say that ${\underbar e}_n$ is an  {\em strict epimorphism} if $\U_{{\underbar e}_{n+1}}$ is a covering of $S$. Any strict epimorphism is an epimorphism, and the converse holds if $S$ is a locally ringed space.
\end{defn}

\begin{defn} Let us consider couples $$(\Lc,{\underbar e}_{n+1})$$ where $\Lc$ is an invertible $\OO_S$-module and ${\underbar e}_{n+1}$ are $n+1$ global sections of $\Lc$. Let us define the equivalence relation:
\[ (\Lc,{\underbar e}_{n+1})\sim(\Lc',{\underbar e}_{n+1}')\] if there exists an isomorphism $\phi\colon \Lc\to\Lc'$ such that $\phi(e_i)=e'_i$. We shall denote $[\Lc,{\underbar e}_{n+1}]$ the equivalence class of a couple $(\Lc,{\underbar e}_{n+1})$.
\end{defn}

It is immediate to see that if $(\Lc,{\underbar e}_{n+1})$ and $(\Lc',{\underbar e}_{n+1}')$ are equivalent, then $ {\underbar e}_{n+1}$ is an epimorphism (resp. a strict epimorphism) if and only if ${\underbar e}_{n+1}'$ is so. Thus, the following definition makes sense
\begin{defn}\label{inv-quotient} An {\em invertible quotient} (resp. {\em a strict invertible quotient}) of $\OO_S^{n+1}$ is an equivalence class $[\Lc,{\underbar e}_{n+1}]$ such that ${\underbar e}_{n+1}$ is an epimorphism (resp. a strict epimorphism).
\end{defn}

 On locally ringed spaces, one has:
\begin{prop}\label{proj-scheme} Let $\Pj^n_{k-\text{\rm sch}}=\Proj k[x_1,\dots,x_{n+1}]$ be the $n$-dimensional projective scheme over $k$. For any locally ringed space $S$ over $k$ one has
\[\aligned  \Hom_{k-\text{\rm loc}}(S,\Pj^n_{k-\text{\rm sch}} )&= \{\text{\rm Invertible quotients of  } \OO_S^{n+1}\}\\ &= \{\text{\rm Strict invertible quotients  } \OO_S^{n+1} \}\\ &=\{ [\Lc,{\underbar e}_{n+1}]: \U_{{\underbar e}_{n+1}} \text{ \rm is a covering}\}.\endaligned \]
\end{prop} Our aim now is to show that an analogous proposition holds for $\Pj^n_k$ after replacing invertible quotients by algebro-topological invertible quotients.

\subsubsection{Algebro-topological invertible 	quotients}  We see here how to reproduce the results of subsection \ref{invertiblequotients} by replacing the sections of the invertible module by what we shall call the algebro-topological sections, in an analogous way as when we replaced algebraic functions by algebro-topological ones.

\begin{defn} An {\em algebro-topological section} of an invertible $\OO_S$-module $\Lc$ is a couple $$\theta=(e_,U)$$  where $e\in\Lc(S)$ and $U$ is an open subset such that $U_e\supseteq U$ (i.e., $e_{\vert U}\colon\OO_U\to\Lc_{\vert U}$ is an isomorphism).  We shall denote
$$U_\theta:=U=U_e\cap U$$ and name it the {\em topological part} of $\theta$, while $e$ is called the {\em algebraic part} of $\theta$.

If $h\colon S'\to S$ is a morphism of ringed spaces, then $h^*\theta:=(h^*(e),h^{-1}(U))$ is an algebro-topological section of $h^*\Lc$.
\end{defn}

\begin{defn} Let us consider couples $$(\Lc,{\underline \theta}_{n+1})$$ where $\Lc$ is an invertible $\OO_S$-module and  ${\underline \theta}_{n+1}=(\theta_1,\dots,\theta_{n+1})$ are $n+1$ algebro-topological sections of $\Lc$. Two couples $(\Lc,{\underline \theta}_{n+1}) $ and $(\Lc',{\underline \theta}_{n+1}')$ are said  to be {\em equivalent}  if they have the same topological part (i.e., $U_{\theta_i}=U_{\theta'_i}$ for every $i$) and they have equivalent algebraic parts (i.e., there exists an isomorphism $\phi\colon\Lc\to\Lc'$ such that $\phi(e_i)=e'_i$ for any $i$, with $e_i,e'_i$ the algebraic parts of $\theta_i,\theta'_i$). We shall denote
\[ [\Lc,{\underline \theta}_{n+1}]\] the equivalence class of a couple $(\Lc,{\underline \theta}_{n+1})$ under this equivalence relation. 
\end{defn} 
  
\noindent{\bf \ \ Notation.} For each $ {\underline \theta}_{n+1}=(\theta_1,\dots,\theta_{n+1})$ we shall denote
\[\U_{{\underline \theta}_{n+1}}:=\{ U_{\theta_1},\dots,U_{\theta_{n+1}}\}.\]
 
\begin{defn}\label{algtopInvQuot} An {\em algebro-topological invertible quotient} of $\OO_S^{n+1}$ is an equivalence class $[\Lc,{\underline\theta}_{n+1}]$ such that $\U_{{\underline\theta}_{n+1}}$ is a covering of $S$. In other words, it is a couple
\[ ([\Lc,{\underbar e}_{n+1}],{ \U}_{n+1})\] where $[\Lc,{\underbar e}_{n+1}]$ is a strict invertible quotient of $\OO_S^{n+1}$ (Definition \ref{inv-quotient}) and ${ \U}_{n+1}$ is a covering by $n+1$ open subsets refining ${ \U}_{{\underbar e}_{n+1}}$ (i.e., $U_i\subseteq U_{e_i}$ for every $i=1,\dots,n$).

If $h\colon S'\to S$ is a morphism of ringed spaces and $[\Lc,{\underline\theta}_{n+1}]$ is an algebro-topological quotient of $\OO_S^{n+1}$, then $[h^*\Lc,h^*({\underline\theta}_{n+1})]$ is an algebro-topological quotient of $\OO_{S'}^{n+1}$. Thus we have a functor of algebro-topological quotients and we shall see now that it is representable by $\Pj_k^n$.
\end{defn}

\begin{thm}\label{Pj^n} {\rm(compare with Proposition \ref{proj-scheme})} For any ringed space $S$ over $k$ one has:
\[\aligned  \Hom_k(S,\Pj^n_k)&=\{ \text{\rm Algebro-topological invertible quotients of }\OO_S^{n+1}\}\\
&=\{ [\Lc,{\underline \theta}_{n+1}]: \U_{{\underline \theta}_{n+1}} \text{ \rm is a covering}\}.\endaligned\]
\end{thm}

\begin{proof} The triplet $\tau=(\OO_{\Pj^n_k}(1),(x_1,\dots,x_{n+1}), \{ U_{\{1\}},\dots,U_{\{n+1\}}\})$ (see Remark \ref{O(1)}) defines an algebro-topologial invertible quotient of $\OO_{\Pj^n_k}^{n+1}$. By Yoneda, this defines a functorial morphism
\[ \aligned \Hom_k(S,\Pj^n)&\to \{ \text{Algebro-topological invertible quotients of }\OO_S^{n+1}\}\\ h&\mapsto [h^*\tau]\endaligned.\] Let us define the inverse. Let $\psi=[\Lc,{\underline \theta}_{n+1}]$ be an algebro-topological invertible quotient of $\OO_S^{n+1}$. Put $\theta_i=(e_i,U_i)$, $i=1,\dots,n+1$.

Let us fix $i\in\{1,\dots,n+1\}$. By definition, ${e_i}_{\vert U_i}\colon\OO_{U_i}\to\Lc_{\vert U_i}$ is an isomorphism. Hence $f_j^i:={e_i}_{\vert U_i}^{-1}\circ {e_j}_{\vert U_i}$ is an algebraic function on $U_i$ that only depends of the isomorphism class $[\Lc,{\underline \theta}_{n+1}]$. We have then  $n+1$ algebro-topological functions on $U_i$
\[ \theta_1^i,\dots, \theta_{n+1}^i,\qquad \theta_j^i:=(f_j^i,U_j\cap U_i)\] that define  a morphism $h_i\colon U_i\to \AAA^{n+1}_k\hskip -1pt {\scriptstyle -\{0\}}$, hence a morphism $$\pi\circ h_i\colon U_i\to\Pj^n_k,$$  where $\pi\colon \AAA^{n+1}_k\hskip -1pt {\scriptstyle -\{0\}}\to\Pj^n_k$ is the quotient morphism. On the intersections $U_i\cap U_j$ one has that ${h_i}_{\vert U_i\cap U_j}= \lambda\cdot {h_j}_{\vert U_i\cap U_j}$, with $\lambda= {e_i}_{\vert U_i\cap U_j}^{-1}\circ {e_j}_{\vert U_i\cap U_j} \in \OO_S(U_i\cap U_j)^\times$. Hence $\pi\circ h_i$ and $\pi\circ h_j$ coincide on $U_i\cap U_j$. This gives a unique morphism $h_\psi\colon S\to\Pj^n_k$ such that ${h_\psi}_{\vert U_i}=\pi\circ h_i$.

 We leave to the reader to check that the assignations $h\mapsto [h^*\tau]$ and $\psi\mapsto h_\psi$ are mutually inverse.
\end{proof}

\begin{rem}\label{rem-projective}[Non-representability, $\pi_0$-representability and comparison with $\Pj^n_{k-\text{sch}}$]\label{norepresentable2} $\,$

(1) The functors $$F,F'\colon \{\text{Ringed spaces over }k\}\to \{\text{Sets}\}$$ defined by
\[\aligned  F(S)&=\{\text{Invertible quotients of }\OO_S^{n+1 } \}\\ F'(S)&=\{\text{Strict invertible quotients of  }\OO_S^{n+1 }  \}\endaligned \]  are not representable (for $n> 0$, since for $n=0$  they are both representable by $\Pj^0_k=(*,k)$). The proof is completely analogous to that of Remark \ref{norepresentable}.

(2) The set $\Hom_k(S,\Pj^n_k)$ has a natural partial order: 
$$[\Lc,{\underline \theta}_{n+1}]\leq [\Lc',{\underline \theta}'_{n+1}]\text{ iff } [\Lc,{\underbar e}_{n+1}]=[\Lc',{\underbar e'}_{n+1}] \text{ and } U_{\theta_i}\subseteq U_{\theta'_i} \text{ for every } i=1,\dots, n+1,$$ 
where $e_i$ (resp. $e'_i$) is the algebraic part of $\theta_i$ (resp. $\theta'_i$).  Thus, $\Hom_k(\quad,\Pj^n_k )$ is a functor from ringed spaces to posets. Now, one has a natural morphism $$\pi\colon \Hom_k(\quad,\Pj^n_k) \to F'$$ defined as: $\pi(S) ([\Lc,{\underline \theta}_{n+1}])= [\Lc,{\underbar e}_{n+1}]$, with $e_i$ the algebraic part of $\theta_i$. Moreover, one has an inclusion $$i(S)\colon F'(S)\to 
 \Hom_k(S,\Pj^n_k)$$ defined as $i(S)([\Lc,{\underbar e}_{n+1} ] )=[\Lc,{\underline \theta}_{n+1}]$, with $\theta_i=(e_i,U_{e_i})$. The composition $\pi(S)\circ i(S)$ is the identity and  $i(S)\circ \pi(S)\leq \id$.
  This means that $F'(S)$, with the discrete topology, is a deformation retract of $\Hom_k(S,\Pj^n_k )$. The elements in the image of $i(S)$  are precisely the maximal elements of $\Hom_k(S,\Pj^n_k )$ (for the partial order) and each $[\Lc,{\underline \theta}_{n+1}]$ is dominated by a unique maximal element. Thus  $F'(S)$ is identified   with both the set of irreducible components  and   the set of connected components of $\Hom_k(S,\Pj^n_k )$. In conclusion $F'$ is $\pi_0$-representable by $\Pj^n_k$:
\[ F'(S)=\pi_0\Hom_k(S,\Pj^n_k ).\]

(3) Let $\Pj^n_{k-\text{sch}}:=\Proj k[x_1,\dots,x_n]$ be the $n$-dimensional projective space over $k$. The triplet 
$$(\OO_{\Pj^n_{k-\text{sch}}}(1), (x_1,\dots, x_{n+1}),\{ U_{x_1},\dots ,U_{x_{n+1}} \})$$
 gives a morphism 
 $$\pi\colon \Pj^n_{k-\text{sch}}\to \Pj^n_k.$$
  For any locally ringed space $S$ over $k$, the induced  morphism (composition with $\pi$) $$\Hom_{k-\text{loc}}(S,\Pj^n_{k-\text{sch}})\to \Hom_k(S,\Pj^n_k)$$ is just the map $[\Lc,\underbar e_{n+1}] \mapsto   [\Lc,\underline \theta_{n+1}]$, with $\theta_i=(e_i,U_{e_i})$. For each morphism of ringed spaces $f\colon S\to \Pj^n_k$ there exists a unique morphism of locally ringed spaces $f_\text{loc}\colon S\to \Pj^n_{k-\text{sch}}$ such that $f\leq \pi\circ f_\text{loc}$.

Finally, $\Pj^n_k$ is the finite ringed space associated to the scheme $\Pj^n_{k-\text{sch}}$ and its standard affine covering $U_{x_1},\dots ,U_{x_{n+1}}$. Thus, we have an equivalence
\[ \Qcoh(\Pj_k^n)\simeq \Qcoh(\Pj^n_{k-\text{sch}}).\] 
\end{rem}

\section{Grassmannian}\label{Grassmannian-section}

In this section we shall construct the grassmannian ringed space $\Grass_k(r,n)$ and show that has a  universal property on ringed spaces analogous to that of the  grassmannian scheme on schemes. We shall first introduce the ringed space of epimorphisms $k^n\to k^r$; the grassmannian will be the quotient of this space by the action of the linear group.
\subsection{Algebro-topological epimorphisms of type $(r,n)$}

\begin{defn} We shall denote by $\HHom_k(k^n,k^r)$ the ringed space representing de functor
\[ S\mapsto \Hom_{\OO_S}(\OO_S^n,\OO_S^r).\]
It is a punctual ringed space and $\OO_{\HHom_k(k^n,k^r)}$ is the polinomial ring in $r\times n$ variables. One has natural isomorphisms $\HHom_k(k,k^n)=\AAA^n_\text{alg}=\HHom(k^n,k)$.
\end{defn}

\begin{defn}\label{Notations r,n} Recall that we denote by $\Delta_{r,n}$  the set of subsets of $r$ elements of $\Delta_n$.  Each $\delta\in\Delta_{r,n}$ induces a  decomposition $k^n=k^r\oplus k^{n-r}$ (if $e_1,\dots,e_n$ is the basis of $k^n$, then $k^r=<e_j>_{j\in\delta}$, $k^{n-r}=<e_j>_{j\notin \delta}$), hence a decomposition $\OO_S^n=\OO_S^r\oplus \OO_S^{n-r}$ for each ringed space $S$ over $k$. For each morphism of $\OO_S$-modules $\phi\colon \OO_S^n\to\OO_S^r$ and each $\delta\in\Delta_{r,n}$ we shall denote 
\[ \phi_\delta\colon \OO_S^r\to\OO_S^r,\quad \phi_{\delta^c}\colon \OO_S^{n-r}\to\OO_S^r\] 
the restrictions of $\phi$ to the factors of the decomposition  $\OO_S^n=\OO_S^r\oplus \OO_S^{n-r}$ induced by $\delta$. 

Finally, let us recall that we denote by $\PP_{r,n}$ the space of parts of $\Delta_{r,n}$ and by $\PP_{r,n}^*$ the open subset of non-empty parts of $\Delta_{r,n}$.
\end{defn}

\begin{defn} Let $\delta\in\Delta_{r,n}$. We define the algebraic function 
\[ \aligned \text{det}_\delta\colon \HHom_k(k^n,k^r) &  \to \HHom_k(k,k)=\AAA^1_{k-\text{alg}}\\ \phi &\mapsto {\det}_\delta\phi:=\wedge^r\phi_\delta \endaligned .\]
\end{defn}

\begin{defn} We say that $\phi\colon\OO_S^n\to\OO_S^r$ is an {\em epimorphism} if it is an stalkwise epimorphism. We say that $\phi$ is an {\em strict epimorphism} if $\wedge^r\phi\colon \OO_S^{\Delta_{r,n}}\to\OO_S$ is a strict epimorphism, i.e. $\U_{\wedge^r\phi}$ is a covering of $S$ (see Definition \ref{epi-strict}). Notice that $\wedge^r\phi$ is given by the collection of algebraic functions $\{\det_\delta \phi\}_{\delta\in\Delta_{r,n}}$.
\end{defn}

Any strict epimorphism is an epimorphism. The converse holds if $S$ is a locally ringed space.

\begin{defn} An {\em algebro-topological epimorphism of type $(r,n)$} on $S$ is a couple $(\phi,\U_{\Delta_{r,n}})$ where:

(1) $\phi\colon\OO_S^n\to\OO_S^r$ is a morphism of modules, 

(2) $\U_{\Delta_{r,n}}=\{ U_\delta\}$ is an open covering of $S$ (indexed by $\Delta_{r,n}$), and 

(3) $\U_{\Delta_{r,n}}\subseteq \U_{\wedge^r\phi}$ (i.e.,   $U_\delta\subseteq U_{\det_\delta\phi}$ for each $\delta\in\Delta_{r,n}$). 

Notice that conditions (2) and (3) imply that $\phi$ is a strict epimorphism. We shall denote
\[ \Epim_{\OO_S}^{\text{alg-top}}(\OO_S^n,\OO_S^r)\] the set of algebro-topological epimorphism of type $(r,n)$ on $S$.
\end{defn}

A morphism of ringed spaces  $h\colon T\to S$ induces a map $$\Epim_{\OO_S}^{\text{alg-top}}(\OO_S^n,\OO_S^r)\to \Epim_{\OO_T}^{\text{alg-top}}(\OO_T^n,\OO_T^r)$$ by sending $(\phi,\U_{\Delta_{r,n}})$ to $(h^*\phi, h^{-1}(\U_{\Delta_{r,n}}))$ (notice that $ \det_\delta(h^*\phi)=h^*(\det_\delta\phi)$).

\begin{thm} The functor 
\[\aligned F\colon \{ \text{\rm Ringed spaces over }k\}&\to {\rm Sets}\\ S&\mapsto F(S)= \Epim_{\OO_S}^{\text{alg-top}}(\OO_S^n,\OO_S^r)\endaligned \] is representable. 
\end{thm}

\begin{proof}   Let us consider the morphisms
\[ \aligned\HHom_k(k^n,k^r)&\to\HHom_k(k^{\Delta_{r,n}},k)=\AAA^{\Delta_{r,n}}_{k-\text{alg}}\\ \phi &\mapsto \wedge^r\phi\endaligned \]
\[ \aligned \AAA^{\Delta_{r,n}}_k\hskip -1pt {\scriptstyle -\{0\}}&\to \AAA^{\Delta_{r,n}}_{k-\text{alg}}\\ (f,\U) &\mapsto f\endaligned \] Then $F$ is representable by the fibred product
\[ \HHom_k(k^n,k^r)\times_{\AAA^{\Delta_{r,n}}_{k-\text{alg}}} \AAA^{\Delta_{r,n}}_k\hskip -1pt {\scriptstyle -\{0\}}.\]
\end{proof}

\begin{defn} We shall denote by $\EEpim_k^{\text{alg-top}}(k^n,k^r)$ the representant of $F$. The proof of the theorem shows that
\[ \EEpim_k^{\text{alg-top}}(k^n,k^r)=(\PP^*_{r,n},\OO_{\EEpim})\]
where $\OO_{\EEpim}$ is the sheaf of rings on $\PP^*_{r,n}$ (see Definition \ref{Notations r,n}) defined  as follows: for each $\underline\delta =\{\delta_1,\dots,\delta_l\}\in \PP_{r,n}^*$, $\OO_{\EEpim,\underline\delta}$ is the localization of the polynomial ring $\OO_{\HHom_k(k^n,k^r)}$ by the algebraic functions $\det_{\delta_1},\dots,\det_{\delta_l}$. For each $\underline\delta\leq \underline{\delta'}$ the morphism $\OO_{\EEpim,\underline\delta}\to \OO_{\EEpim,\underline{\delta'}}$ is the natural one.

The universal object is the couple $(\Phi,\U_{\Delta_{r,n}})$ where:

(1) $\Phi\colon \OO_{\EEpim}^n\to\OO_{\EEpim}^r$ is a strict epimorphism, which is the pullback of the universal object of $\HHom_k(k^n,k^r)$ by the natural morphism $\EEpim_k^{\text{alg-top}}(k^n,k^r)\to\HHom_k(k^n,k^r)$.

(2) $\U_{\Delta_{r,n}}=\{ U_\delta\}$ is the universal covering of $\PP^*_{r,n}$; that is, $U_\delta=U_{\{\delta\}}.$ 

One has $\U_{\Delta_{r,n}}=\U_{\wedge^r\Phi}$.
\end{defn}

\subsection{Linear group}\label{GL}

\begin{defn} The functor $$\aligned \bGL_r\colon\{\text{Ringed spaces over } k \}&\to \{\text{Groups}\}
\\ S&\mapsto \Aut_{\OO_S}(\OO_S^r)=\operatorname{GL}_r(\OO_S(S))\endaligned$$
is representable by the punctual ringed space over $k$
\[ \GL_r:=(*,k[x_{ij}]_{\text{det}})\] where $x_{ij}$ are $r\times r$ variables and $\text{det}$ is the determinant of the matrix $(x_{ij})$.
\end{defn}

\begin{defn} An {\em action} of $\GL_r$ on a ringed space $S$ over $k$ is a morphism of ringed spaces
\[\mu\colon \GL_r\times _k S\to S\] satisfying the standard axioms of an action. Since $\GL_r$ is a punctual ringed space, the action $\mu$ is equivalent to a morphism of sheaves of algebras
\[ \mu_\#\colon \OO_S\to\OO_S\otimes_k k[x_{ij}]_\text{det}.\]  
The {\em sheaf of invariants} $\OO_S^{\GL_r}$ is defined by:
\[ \OO_S^{\GL_r}:=\Ker (\mu_\# - i)\] where $i\colon\OO_S\to \OO_S\otimes_k k[x_{ij}]_\text{det}$ is the natural inclusion (i.e., $i(a)=a\otimes 1$).  We shall define the {\em quotient }$S/\GL_r$ as the ringed space
\[  S/\GL_r:=(S,\OO_S^{\GL_r})\] and one has a natural morphism $S\to S/\GL_r$ which is $\GL_r$-equivariant (considering the trivial action of $\GL_r$ on $S/\GL_r$).
\end{defn}

\begin{ejems}\label{Gl-homomorphisms}  \begin{enumerate}\item The left action of $\GL_r$  on itself is defined by the multiplication law of $\GL_r$: $\mu(g,g')=g\circ g'$. One has that
\[\GL_r/\GL_r=(*,k)\] since $(k[x_{ij}]_\text{det})^{\GL_r}=k$. We shall always assume this action of $\GL_r$ on itself unless otherwise specified.

\item Let $S_{\text{trivial}}$ be a ringed space over $k$ endowed with the trivial action of $\GL_r$. Then  
\[ (\GL_r\times_kS_{\text{trivial}})/\GL_r = S_{\text{trivial}}.\] If $S$ is a ringed space with an action of $\GL_r$, then one has an equivariant isomorphism $\GL_r\times_k S\overset\sim\to \GL_r\times_k S_{\text{trivial}}$, given by $(g,s)\mapsto (g,g^{-1}\cdot s)$, and then
\[ (\GL_r\times_k S)/\GL_r\simeq S_{\text{trivial}}.\]
\item $\GL_r$ acts on $\HHom_k(k^n,k^r) $ by left composition: $\mu(g,\phi)=g\circ\phi$.
\end{enumerate}
\end{ejems}

\subsection{Grassmannian}

\begin{defn} One has a natural action of $\GL_r$ on $\EEpim_k^{\text{alg-top}}(k^n,k^r)$, namely
\[\aligned \GL_k\times_k  \EEpim_k^{\text{alg-top}}(k^n,k^r) &\to  \EEpim_k^{\text{alg-top}}(k^n,k^r) \\ (g,(\phi,\U))&\mapsto (g\circ\phi,\U)\endaligned\] (notice that $\U_{\wedge^r(g\circ\phi)}=\U_{\wedge^r\phi}$). We define the {\em grassmannian  ringed space} $\Grass_k(r,n)$ as the quotient
\[ \Grass_k(r,n):=\EEpim_k^{\text{alg-top}}(k^n,k^r)/\GL_r.\]
By definition, 
\[ \Grass_k(r,n)=(\PP^*_{r,n},\OO_{\Grass_k(r,n)})\] 
where  $\OO_{\Grass_k(r,n)}=\OO_{\EEpim}^{\GL_r}$.
\end{defn}

Let us see the local structure of the quotient morphism $\pi\colon \EEpim_k^{\text{alg-top}}(k^n,k^r)\to\Grass_k(r,n)$.

\begin{defn} $\PP^*_{r,n}$ is covered by the open subsets $U_{\{\delta\}}$, $\delta\in\Delta_{r,n}$. We shall denote by $\EEpim_k^\delta(k^n,k^r)$ the open ringed subspace of $\EEpim_k^{\text{alg-top}}(k^n,k^r)$ whose underlying topological space is $U_{\{\delta\}}$:
\[ \EEpim_k^\delta(k^n,k^r):=(U_{\{\delta\}},{\OO_{\EEpim}}_{\vert U_{\{\delta\}}}).\] 
Analogously, $\Grass_k^\delta(r,n)$ denotes the open ringed space of $\Grass_k(r,n)$ with underlying topological space $U_{\{\delta\}}$.
\end{defn}

\begin{rem} $\EEpim_k^\delta(k^n,k^r)$ represents the subfunctor of $\EEpim_k^{\text{alg-top}}(k^n,k^r)$ of the couples $(\phi,\U_{\Delta_{r,n}})\in \Epim_{\OO_S}^{\text{alg-top}}(\OO_S^n,\OO_S^r)$ such that $U_\delta=S$. In particular, it is a $\GL_r$-invariant open ringed subspace of $\EEpim_k^{\text{alg-top}}(k^n,k^r)$ and
\[ \EEpim_k^\delta(k^n,k^r)/\GL_k =\Grass_k^\delta(r,n).\]
\end{rem}

\begin{prop} One has a $\GL_r$-equivariant isomorphism
\[ \EEpim_k^\delta(k^n,k^r)\simeq \GL_r\times_k X_\delta\] where $X_\delta$ is a ringed space endowed with the trivial action of $\GL_r$ (the explicit description of $X_\delta$ shall be given in the proof). Consequently
\[\Grass_k^\delta(r,n)\simeq X_\delta \] and the quotient morphism $\pi\colon\EEpim_k^{\text{alg-top}}(k^n,k^r)\to \Grass_k(r,n)$ is locally trivial: $$\pi^{-1}(\Grass^\delta_k(r,n))\simeq \GL_r\times_k \Grass^\delta_k(r,n).$$
\end{prop}

\begin{proof} Let us consider the algebraic function $\det_\delta\colon \HHom_k(k^n,k^r)\to\AAA^1_\text{alg}$ and let us denote $\HHom^\delta_k(k^n,k^r)=\HHom_k(k^n,k^r)\times_{\AAA^1_\text{alg}}\G_m$. It represents the subfunctor of $\HHom_k(k^n,k^r)$ given by the morphisms $\phi\colon\OO_S^n\to\OO_S^r$ such that $\phi_\delta$ is an isomorphism. $\GL_r$ acts on $\HHom^\delta_k(k^n,k^r)$ by left composition. One has an  isomorphism  
\[ \aligned\HHom_k^\delta(k^n,k^r)&\to \GL_r\times_k\HHom_k(k^{n-r},k^r) \\ \phi&\mapsto (\phi_\delta,\phi_\delta^{-1}\circ \phi_{\delta^c})\\ g+g \cdot\psi &\leftarrow (g,\psi)\endaligned\] which is $\GL_r$-equivariant, considering the trivial action on $\HHom_k(k^{n-r},k^r)$.
For every $\phi=g+g \cdot\psi$, the condition $U_{\delta'}\subseteq U_{\det_{\delta'}\phi}$ is equivalent to the condition $U_{\delta'}\subseteq U_{\det_{\delta'}(1+\psi)}$. 
Thus we have a bijection between the set of couples $(\phi,\{U_{\delta'}\}_{\delta'\in\Delta_{r,n}})\in\EEpim^\delta_k(k^n,k^r)$ and the set of triplets $(g,\psi,\{\U_{\delta'}\}_{\delta'\neq \delta})$ such that $U_{\delta'}\subseteq U_{\det_{\delta'}(1+\psi)}$, whose inverse maps  a triplet $(g,\psi,\{\U_{\delta'}\}_{\delta'\neq \delta})$   to the couple $(g+g\cdot\psi,\{\U_{\delta'}\}_{\delta'\neq \delta}\cup\{U_\delta=S\})$. Hence, if we consider the morphisms
 \[ \aligned \HHom_k(k^{n-r},k^r)&\to \AAA^{\Delta_{r,n}{\scriptscriptstyle-\{\delta\}}}_{k-\text{alg}} \\ \psi&\mapsto ({\det}_{\delta'}(1+\psi))_{\delta'\neq\delta}\endaligned\]
 \[ \aligned \AAA^{\Delta_{r,n}{\scriptscriptstyle-\{\delta\}}}_{k} &\to \AAA^{\Delta_{r,n}{\scriptscriptstyle-\{\delta\}}}_{k-\text{alg}}\\ (f,\U)&\mapsto f \endaligned\] it suffices to define $X_{\delta}$ as their fibred product
 \[ X_{\delta}:= \HHom_k(k^{n-r},k^r)\times_{ \AAA^{\Delta_{r,n}{\scriptscriptstyle-\{\delta\}}}_{k-\text{alg}} }\AAA^{\Delta_{r,n}{\scriptscriptstyle-\{\delta\}}}_{k}.\]
\end{proof} 
\subsection{Algebro-topological quotients of rank $r$}

Let $\E$ be a locally free module of rank $r$ on a ringed space $S$. Let $\phi\colon \OO_S^n\to\E$ be a morphism of modules and $\wedge^r\phi\colon \OO_S^{\Delta_{r,n}}\to\wedge^r\E$ the induced morphism. Notice that $\wedge^r\E$ is an invertible module.
 
\medskip
\noindent{\bf\ \ Notation}. For each $\delta\in\Delta_{r,n}$, let $\phi_\delta\colon\OO_S^r\to\E$ be the restriction of $\phi$ to the factor $\OO_S^r$ of the decomposition $\OO_S^n=\OO_S^r\oplus\OO_S^{n-r}$ induced by $\delta$. Then $\wedge^r\phi\colon \OO_S^{\Delta_{r,n}}\to\wedge^r\E$ is the collection of sections $\wedge^r\phi_\delta$ of $\wedge^r\E$. 

\begin{defn}\label{epi-strict-r} We say that $\phi$ is an {\em epimorphism} if it is a stalkwise epimorphism, This is equivalent to say that $\wedge^r\phi$ is an epimorphism. We say that $\phi$ is a  {\em strict epimorphism} if $\wedge^r\phi$ is a strict epimorphism, i.e. if $\U_{\wedge^r\phi}$ is a covering (see Definition \ref{epi-strict}). Any strict epimorphism is an epimorphism, and the converse holds if $S$ is a locally ringed space.
\end{defn}

\begin{defn} Let us consider couples $$(\E,\phi)$$ where $\E$ is a locally free module of rank $r$ and $\phi\colon\OO_S^n\to\E$ is a morphism of modules.  Let us define the equivalence relation:
\[ (\E, \phi)\sim(\E',\phi')\Leftrightarrow\text{ there exists an isomorphism }\psi\colon \E\to\E' \text{ such that }\psi\circ\phi=\phi' \]  We shall denote $[\E,\phi]$ the equivalence class of a couple $(\E,\phi)$ under this equivalence relation.
\end{defn}

It is immediate to see that if $(\E,\phi)$ and $(\E',\phi')$ are equivalent, then $ \phi$ is an epimorphism (resp. a strict epimorphism) if and only if $\phi'$ is so. Thus, the following definition makes sense.
\begin{defn} A  {\em locally free quotient of rank   $r$ of $\OO_S^n$} (resp. an {\em   strict locally free  quotient  of rank $r$ of $\OO_S^{n}$}) is an equivalence class $[\E,\phi]$ such that $\phi$ is an epimorphism (resp. a strict epimorphism).
\end{defn}

\begin{prop}\label{grass-scheme} Let $\Grass_{k-\text{\rm sch}}(r,n)$ be the grassmannian scheme over $k$. For any locally ringed space $S$ over $k$ one has
\[\aligned  \Hom_{k-\text{\rm loc}}(S,\Grass_{k-\text{\rm sch}}(r,n)) )&= \{\text{\rm Locally free quotients of  rank $r$ of  } \OO_S^{n}\}\\ &= \{\text{\rm Strict locally free quotients of  rank $r$ of  } \OO_S^{n} \}\\ &=\{ [\E,\phi]: \U_{\wedge^r\phi} \text{ \rm  is a covering}\}.\endaligned \]
\end{prop}

\begin{defn}\label{algtopQuot} An {\em  algebro-topological quotient of rank $r$ of $\OO_S^n$} is a couple
\[ ([\E,\phi],\U_{\Delta_{r,n}})\] where:

(1) $[\E,\phi]$ is an equivalence class (with $\E$ a locally free module of rank $r$ and $\phi\colon\OO_S^n\to\E$ a morphism of modules),

(2) $\U_{\Delta_{r,n}}=\{ U_\delta\}$ is a covering of $S$, indexed by $\Delta_{r,n}$, and

(3) $\U_{\Delta_{r,n}}\leq \U_{\wedge^r\phi}$, that is, $U_\delta\subseteq U_{\wedge^r\phi_\delta}$ for every $\delta\in\Delta_{r,n} $.

$[\E,\phi]$ is called the {\em algebraic part} of the algebro-topological quotient and $\U_{\Delta_{r,n}}$ the {\em topological part}. Notice that conditions (2) and (3) imply that $[\E,\phi]$ is a strict locally free quotient of rank $r$ of $\OO_S^n$.
\end{defn} 

If $h\colon S'\to S$ is a morphism of ringed spaces, and $\psi=([\E,\phi],\U_{\Delta_{r,n}})$ is an algebro-topological quotient of rank $r$ of $\OO_S^n$ then 
\[ h^*\psi:=  (h^*[\E,\phi],h^{-1}(\U_{\Delta_{r,n}}))\] is an algebro-topological quotient of rank $r$ of $\OO_{S'}^n$.

\subsection{The universal bundle on $\Grass_k(r,n)$}
 
Let us denote, for short,  $$\EEpim=\EEpim_k^{\text{alg-top}}(k^n,k^r)$$ and let $(\Phi,\U_{\Delta_{r,n}})$ be the universal object; recall that 
\[\Phi\colon \OO_{\EEpim}^n\to\OO_{\EEpim}^r\] is a strict epimorphism, $\U_{\Delta_{r,n}}=\{ U_{\{\delta\}}\}$ and  $U_{\{\delta\}}=  U_{\det_\delta\Phi}$. 

The epimorphism $\Phi$ may be viewed as a morphism over $\EEpim$:
\[ \xymatrix{ \EEpim\times_k\AAA^n_{k-\text{alg}}\ar[rr]^\Phi\ar[rd] & & \EEpim\times_k\AAA^r_{k-\text{alg}}\ar[ld]
\\ & \EEpim & }\] which is described by the morphism of functors $$((\phi,\U),f)\mapsto ((\phi,\U),\phi\circ f)$$
(we are thinking $f\in \AAA^n_{k-\text{alg}}=\HHom_k(k,k^n)$ and analogously for $\AAA^r_{k-\text{alg}}$)
This triangle is $\GL_r$-equivariant, acting trivially on $\AAA^n_{k-\text{alg}}$ and by the left on  $\AAA^r_{k-\text{alg}}=\HHom_k(k,k^r)$. Moreover,   each $\delta$ induces a  morphism over $\EEpim$ 
$$1\times\delta\colon \EEpim\times_k\AAA^r_{k-\text{alg}} \to \EEpim\times_k\AAA^n_{k-\text{alg}}$$ which is $\GL_r$-equivariant (acting trivially on $\AAA^r_{k-\text{alg}}$) and whose composition with $\Phi$ is an isomorphism over $\EEpim^\delta$.

Taking the quotient by $\GL_r$ one obtains  morphisms  over $\Grass_k(r,n)$:
\[\xymatrix{  \Grass_k(r,n)\times_k \AAA^r_{k-\text{alg}}\ar[r]^{\overline{1\times\delta}}\ar[rd] & \Grass_k(r,n)\times_k \AAA^n_{k-\text{alg}}\ar[r]^{\bar\Phi}  \ar[d] &   (\EEpim\times_k\AAA^r_{k-\text{alg}})/\GL_r\ar[ld]\\ &\Grass_k(r,n) & } \] and the composition $\overline\Phi\circ\overline{1\times\delta}$ is an isomorphism over $\Grass^\delta_k(r,n)$.

\begin{defn} Let us denote $\operatorname{Q}=(\EEpim\times_k\AAA^r_{k-\text{alg}})/\GL_r$ and let  $\Q$ be the sheaf of sections of $\operatorname{Q}\to\Grass_k(r,n)$. Then  $\Q$ is a locally free $\OO_{\Grass_k(r,n)}$-module of rank $r$. The morphism $\overline\Phi$ may be viewed as a  morphism of modules $\overline\Phi\colon \OO_{\Grass_k(r,n)}^n\to\Q$ and it satisfies $U_{\{\delta\}}=U_{\wedge^r\overline\Phi_\delta}$. In conclusion
\[ ([\Q,\overline\Phi],\U_{\Delta_{r,n}})\] is an algebro-topological quotient of rank $r$ of $\OO_{\Grass_k(r,n)}^n$. This is called the {\em universal algebro-topological quotient} on the grassmannian.
\end{defn}
 
\begin{thm}\label{Grass-thm}{\rm (compare with Proposition \ref{grass-scheme})} $\Grass_k(r,n)$ represents the functor of algebro-topological quotients of rank $r$ of the free module of rank $n$. That is
\[ \Hom_k(S,\Grass_k(r,n))=\{ \text{\rm Algebro-topological quotients of rank } r \text { \rm  of }\OO_S^n\}.\]
\end{thm}

\begin{proof} By Yoneda, the universal algebro-topological quotient $u=([\Q,\Phi],\U_{\Delta_{r,n}})$  defines a functorial morphism
\[\aligned \Hom_k(S,\Grass_k(r,n))&\to \{ \text{\rm Algebro-topological quotients of rank } r \text { \rm  of }\OO_S^n\}\\ h &\mapsto h^*u.\endaligned\] Let us construct the inverse. Let $\theta= ([\E,\phi],\U_{\Delta_{r,n}})]$ be an algebro-topological quotient of rank $r$ of $\OO_S^n$. For each $\delta\in \Delta_{r,n}$, $\phi_\delta\colon\OO_S^r\to\E$ is an isomorphism on $U_\delta$. Hence we have a strict epimorphism
\[ \OO_{U_\delta}^n\overset{{\phi_\delta}_{\vert U_\delta}^{-1}\circ \phi_{\vert U_\delta}}\longrightarrow \OO_{U_\delta}^r\] which, together with the covering $U_\delta\cap \U_{\Delta_{r,n}}$, defines a morphism
\[ h_\delta\colon U_\delta\to\EEpim_k^{\text{alg-top}}(k^n,k^r)\] and, then a morphism $\pi\circ h_\delta\colon U_\delta\to\Grass_k(r,n)$. If we take a different $\delta'$, then $h_\delta$ and $h_{\delta'}$ differ, on $U_\delta\cap U_{\delta'}$, in an element of $\GL_r(\OO_S(U_\delta\cap U_{\delta'}))$; hence $\pi\circ h_\delta$ and $\pi\circ h_{\delta'}$ agree on $U_\delta\cap U_{\delta'}$. This defines a morphism $h_\theta\colon S\to\Grass_k(r,n)$ (unique such that ${h_\theta}_{\vert U_\delta}=\pi\circ h_\delta$).
We leave the reader to check that the assignations $h\mapsto h^*u$, $\theta\mapsto h_\theta$ are mutually inverse.
\end{proof}

\begin{rem}[Non-representability, $\pi_0$-representability and comparison with $\Grass_{k-\text{sch}}(r,n)$]\label{norepresentable3} $\,$
We reproduce Remarks \ref{re-punctured} and \ref{rem-projective} for the grassmannian. The proofs are completely analogous.

(1) The functors $$F,F'\colon \{\text{Ringed spaces over }k\}\to \text{Sets}$$ defined by
\[\aligned  F(S)&=\{\text{Locally free quotients of rank $r$ of }\OO_S^{n } \}\\ F'(S)&=\{\text{Strict locally free quotients of rank $r$ of   }\OO_S^{n }  \}\endaligned \]  are not representable (for $n> r$, since for $n=r$  they are both representable by $(*,k)$).

(2) The set $\Hom_k(S,\Grass_k(r,n))$ has a natural partial order (the partial order induced by inclusion of the topological parts and equality of the algebraic parts); one has a natural morphism $$\pi\colon \Hom_k(\quad,\Grass_k(r,n)) \to F'$$ and  $F'$ is $\pi_0$-representable by $\Grass_k(r,n)$:
\[ F'(S)=\pi_0\Hom_k(S,\Grass_k (r,n)).\]

(3) Let $\Grass_{k-\text{sch}}(r,n)$ be the grassmannian scheme representing (in the category of $k$-schemes) the functor of locally free quotients of rank $r$ of the free module of rank $n$. Let $\Phi_\text{sch}\colon\OO_{\Grass_{k-\text{sch}}(r,n)}^n\to \Q_\text{sch}$ be the universal quotient  and $\{\Grass^\delta_{k-\text{sch}}(r,n)\}_{\delta\in\Delta_{r,n}}$ the standard affine open covering ($\Grass^\delta_{k-\text{sch}}(r,n)$ is the set of points where $(\Phi_\text{sch})_\delta$ is an isomorphism). These data define a morphism  
$$\pi\colon \Grass_{k-\text{sch}}(r,n)\to \Grass_k(r,n).$$ For any locally ringed space $S$ over $k$, the induced  morphism $$\Hom_{k-\text{loc}}(S,\Grass_{k-\text{sch}}(r,n))\to \Hom_k(S,\Grass_k(r,n))$$ is just the map $[\E,\phi] \mapsto   ([\E,\phi],\U_{\wedge^r\phi})$. For each morphism of ringed spaces $f\colon S\to \Grass_k(r, n)$ there exists a unique morphism of locally ringed spaces $f_\text{loc}\colon S\to \Grass_{k-\text{sch}}(r, n)$ such that $f\leq \pi\circ f_\text{loc}$.

Finally, $\Grass_k(r,n)$ is the finite ringed space associated to the scheme $\Grass_{k-\text{sch}}(r,n)$ and the affine covering $\{\Grass^\delta_{k-\text{sch}}(r,n)\}_{\delta\in\Delta_{r,n}}$. Thus, we have an equivalence
\[ \Qcoh(\Grass_k(r,n))\simeq \Qcoh(\Grass_{k-\text{sch}}(r,n)).\] 
\end{rem}

\subsection{The Pl\"ucker embedding}

We have already considered the morphism
\[ \aligned\HHom_k(k^n,k^r)&\to\HHom_k(k^{\Delta_{r,n}},k)=\AAA^{\Delta_{r,n}}_{k-\text{alg}}\\ \phi &\mapsto \wedge^r\phi\endaligned. \] In the same way we may define the Pl\"ucker-type morphism
\[ \aligned P\colon\EEpim_k(k^n,k^r)&\to\EEpim_k(k^{\Delta_{r,n}},k)=\AAA^{\Delta_{r,n}}_k \hskip -1pt {\scriptstyle -\{0\}}\\ (\phi,\U) &\mapsto (\wedge^r\phi,\U).\endaligned \] 
 $\GL_r$ acts on $\EEpim_k(k^n,k^r)$ by the left and $\G_m$ acts on $\AAA^{\Delta_{r,n}}_k \hskip -1pt {\scriptstyle -\{0\}}$ by homotheties. Moreover one has the morphism of groups
 \[ \aligned\operatorname{Det}\colon \GL_r&\to\G_m\\ g&\mapsto \det g.\endaligned\]
 
$P$ is equivariant with respect to  $\operatorname{Det}$. That is, one has 
 \[ P(g\cdot(\phi,\U))=(\det g)\cdot P(\phi,\U).\] Thus, one obtains a morphism between the quotients
 \[\overline P\colon \Grass_k(r,n)\to \Pj^{\binom nr -1}_k\] which is called the {\em Pl\"ucker embedding of the grassmannian}. Its functorial description is
 \[\overline P([\E,\phi],\U)= ([\wedge^r\E,\wedge^r\phi],\U).\] Notice that $P$ and $\overline P$ are the identity on the underlying topological spaces. 
 
The Pl\"ucker embedding  $\overline P$ is compatible with the usual Pl\"ucker embedding of schemes $\overline P_\text{sch}\colon \Grass_{k-\text{sch}}\to \Pj^{\binom nr -1}_{k-\text{sch}} $;   that is, one has a commutative diagram
\[\xymatrix{\Grass_{k-\text{sch}}(r,n)\ar[r]\ar[d]_{\overline P_\text{sch}} & \Grass_k(r,n)\ar[d]^{\overline P} \\ \Pj^{\binom nr -1}_{k-\text{sch}} \ar[r] & \Pj^{\binom nr -1}_k}\] since, for any locally ringed space $S$ over $k$, the diagram
\[\xymatrix{\Hom_{k-\text{loc}}(S,\Grass_{k-\text{sch}}(r,n)) \ar[r]\ar[d] & \Hom_k(S,\Grass_k(r,n))\ar[d]\\ 
 \Hom_{k-\text{loc}}(S,\Pj_{k-\text{sch}}^{\binom nr -1})  \ar[r] & \Hom_k(S,\Pj^{\binom nr -1}_k)}\] is commutative.
 
\section{Stanley-Reisner theory}\label{Stanley-Reisner}

For all this section, $k$ denotes a field. Let us recall some notations: $\PP_n$ is the set of subsets of $\Delta_n=\{ 1,\dots,n\}$, with the topology associated to the partial order given by inclusion, and $\PP^*_n$ is the set of non-empty subsets of $\Delta_n$, which is an open subset of $\PP_n$.   For each $i\in \Delta_n$, $U_{\{i\}}$ denotes the smallest open subset of $\PP_n$ containing $\{ i\}$. Thus, $U_{\{i\}}$ is the set of subsets of $\Delta_n$ containing the element $i$ and $\PP^*_n$ is covered by $U_{\{1\}}, \dots, U_{\{n\}}$.
 
\begin{defn}  By a {\em simplicial complex} we shall mean a finite abstract simplicial complex, i.e., a non empty closed subset $K$ of $\PP_n $, for some $n$. By a {\em reduced simplicial complex} we mean a non empty closed subset $K^*$ of  $\PP^*_n$. Obviously, any simplical complex $K$ produces a reduced one $K^*=K-\{*\}$ (except when $K=\{*\}$) and conversely, a reduced simplicial complex $K^*$ gives a simplicial complex $K=K^*\cup \{*\}$.\end{defn} 

\begin{defn}  Let $K$ be a simplicial complex. For each $p\in K$ we shall denote by $U_p$ (resp. $C_p$) the smallest open (resp. closed) subset of $K$ containing $p$. In other words
\[ U_p=\{ q\in K: q\geq p\},\qquad C_p=\{ q\in K: q\leq p\}. \] The {\em link of $p$} is
\[ \operatorname{link}_K(p)=\{ q\in K^*:  U_p\cap U_q\neq\emptyset\text{ and } 
C_p\cap C_q=\{*\}\}. \] It is a closed subset of $K^*$. If $K^*$ is a reduced simplicial complex, for each $p\in K^*$ we shall denote
\[\operatorname{link}_{K^*}(p)=\operatorname{link}_K(p) \qquad (K=K^*\cup\{*\}).\]
A simplicial complex $K$ is called {\em Cohen-Macaulay} (over $k$) if for any $p\in K$
\[ \widetilde H_i(\operatorname{link}_K(p),k)=0, \text{ for  every }i<\dim \operatorname{link}_K(p),\] where $\widetilde H_i$ denotes reduced homology.
A reduced simplicial complex $K^*$ is called {\em Cohen-Macaulay} (over $k$) if it is pure and for any $p\in K^*$
\[ \widetilde H_i(\operatorname{link}_{K^*}(p),k)=0, \text{ for  every }i<\dim \operatorname{link}_{K^*}(p).\]
\end{defn}

Let us consider the $n$-dimensional affine $k$-scheme $\AAA^n_{k-\text{sch}}=\Spec k[x_1,\dots,x_n]$. The open subsets $U_{x_1}\dots,U_{x_n}$ define a continuous surjective map
\[ \pi\colon \AAA^n_{k-\text{sch}}\to \PP_n.\]  

\begin{defn} Let $K$ be a closed subset of $\PP_n$. Then $\pi^{-1}(K)$ is a closed subset of $\AAA^n_{k-\text{sch}}$, which defines a---unique---reduced closed subscheme $\pi^{-1}(K)_\text{sch}$ of $\AAA^n_{k-\text{sch}}$. This closed subscheme is defined by a radical ideal $I$ of $k[x_1,\dots,x_n]$. This is the {\em Stanley-Reisner ideal} associated to the abstract simplicial complex $K$, and it is ususally denoted by $I_K$. Thus
\[ \pi^{-1}(K)_\text{sch}=\Spec k[x_1,\dots,x_n]/I_K  \] and $k[x_1,\dots,x_n]/I_K$ is the {\em Stanley-Reisner ring} (over $k$) of $K$.
\end{defn}

 Reisner Theorem states that $K$ is a Cohen-Macaulay simplical complex if and only if $\pi^{-1}(K)_\text{sch}$ is a Cohen-Macaulay scheme.
 
 Now let us consider the punctured scheme $\AAA^n_k\hskip -1pt {\scriptstyle -\{0\}}$; the covering $U_{x_1},\dots,U_{x_n}$ defines a continuous surjective map
\[ \pi\colon \AAA^n_{k-\text{sch}}\hskip -1pt {\scriptstyle -\{0\}}\to \PP_n^*.\] 

As before, if $K^*$ is a closed subset of $\PP_n^*$, then $\pi^{-1}(K^*)$ is a closed subset of $\AAA^n_k\hskip -1pt {\scriptstyle -\{0\}}$ which defines a - unique - reduced closed subscheme $\pi^{-1}(K^*)_\text{sch}$ of $\AAA^n_k\hskip -1pt {\scriptstyle -\{0\}}$. Then one has (see \cite[Exercise 5.3.15]{BH}) 
\[ K^*\text{ is Cohen-Macaulay } \Leftrightarrow  \pi^{-1}(K^*)_\text{sch} \text{ is a Cohen-Macaulay scheme}.\]
 Now, let us take the projective scheme $\Pj^{n-1}_{k-\text{sch}}=\Proj k[x_1,\dots,x_n]$. The standard affine covering $U_{x_1},\dots ,U_{x_n}$ of $\Pj^{n-1}_{k-\text{sch}}$  defines a continuous surjective map
\[ \overline\pi\colon \Pj^{n-1}_k \to \PP_n^* \] and one has a commutative diagram
\[ \xymatrix{ \AAA^n_k\hskip -1pt {\scriptstyle -\{0\}}\ar[r]^\pi\ar[d] & \PP_n^*\\ \Pj^{n-1}_{k-\text{sch}}\ar[ur]_{\overline\pi} & }\] A closed subset $K^*$ of $\PP^*_n$ defines a reduced closed subscheme ${\overline\pi}^{-1}(K^*)_\text{sch}$ of $\Pj^{n-1}_{k-\text{sch}}$ and one has:
\[ K^*\text{ is Cohen-Macaulay } \Leftrightarrow  {\overline\pi}^{-1}(K^*)_\text{sch} \text{ is a Cohen-Macaulay scheme}\] since ${\overline\pi}^{-1}(K^*)_\text{sch}$ is Cohen-Macaulay if and only if $\pi^{-1}(K^*)_\text{sch}$ is so.

\subsection{A general problem}

Let $X$ be a smooth $k$-scheme and let $U_1,\dots, U_n$ be affine open (non empty) subsets. Let us denote $D_i=X-U_i$ the complementary closed subset, with the reduced scheme structure, which is a reduced positive divisor. Let 
\[\pi\colon X\to \PP_n\] be the continuous map associated to $U_1,\dots,U_n$. Let us asume that $\pi$ is surjective and let  $K$ be a closed subset of $\PP_n$. Let $\pi^{-1}(K)_\text{sch}$ be  the  associated reduced closed subscheme of $X$. In this case, the Reisner ideal is
\[ \mathcal I_K=\underset{q\notin K}\sum \mathcal I_{D_q},\text{ with } D_q= \sum_{i\in q} D_i .\]

The natural question is how the Cohen-Macaulicity of $K$ is related to that of $\pi^{-1}(K)_\text{sch}$. Does Reisner criterion hold? Some obvious transversality hypothesis on the divisors $D_i$ are necessary to be assumed: since for any $p\in\PP_n$, $C_p$ is a Cohen-Macaulay simplicial complex, one should ask $\pi^{-1}(C_p)_\text{sch}$ to be Cohen-Macaulay; now  
\[ \pi^{-1}(C_p)_\text{sch}=\underset {i\notin p}\bigcap D_i;\] Hence, if we want $\pi^{-1}(C_p)_\text{sch}$ to be Cohen-Macaulay for any $p\in\PP_n$, we must ask the divisors $D_i$ to meet transversally: for any $x\in X$, the divisors $D_i$ passing through $x$ meet transversally at $x$ (i.e., the local equations of these $D_i$ at $x$ form a regular sequence). Now the question is: if $D_i$ meet transversally, does Reisner criterion hold? Very probably, the answer  is no with this generality. Some natural questions would be:

(1) what further hypothesis (on $X$ and on the divisors $D_i$) are necessary to be assumed such that Reisner criterion   holds?

(2)  what simplicial complexes $K$ of $\PP_n$ satisfy that $\pi^{-1}(K)_\text{sch}$ is Cohen-Macaulay for any $X$ and any transversal divisors $D_1,\dots D_n$ (assuming for instance that $n=\dim X$)?

In our opinion, the natural way to relate the Cohen-Macaulicity of $K$ to that of $\pi^{-1}(K)_\text{sch}$ is via duality, at least for proper schemes, in the following way. Let $D_K$ be the dualizing complex of $K$, which is a complex of sheaves of $k$-vector spaces on $K$. On the other hand, let us denote $X_K=\pi^{-1}(K)_\text{sch}$ and let $D_{X_K}$ be the dualizing complex of $X_K$, which is a complex of quasi-coherent (in fact, coherent) $\OO_{X_K}$-modules on $X$. Let $\OO=\pi_*\OO_{X_K}$, which is a sheaf of $k$-algebras on $K$. Then, the category of quasi-coherent modules on $X_K$ is equivalent to the category of quasi-coherent modules on the finite ringed space $(K,\OO)$.  In particular, the dualizing complex $D_{X_K}$ agrees (via this equivalence) with the dualizing complex of $(K,\OO)$. Now, one has a natural morphism of ringed spaces
\[ I\colon (K,\OO)\to (K,k)\] which is given by the identity at the topological level and the natural morphism of sheaves $k\to\OO$. Despite the validity of Reisner criterion, we think that it is worthy to study how the morphism $I$ relates the dualizing complex   $D_K$ (which is the dualizing complex in the category of sheaves of  modules  on $(K,k)$) to the dualizing complex of $(K,\OO)$ (which is the dualizing complex in the category of quasi-coherent modules on $(K,\OO)$). This will be done in a future paper.

\section{Non schematic ringed spaces}	

This section aims to give elementary examples of functors on ringed spaces that are representable by finite ringed spaces which are not schematic (see \cite{Sancho4}) and hence they are not the finite ringed space associated to some affine covering of a scheme (in contrast with $\AAA^n_k$, $\Pj^n_k$ or $\Grass_k(r,n)$). We leave the proofs to the reader, since they are quite elementary.

One of the most elementary and natural functors to be considered on ringed spaces, by the very definition of ringed space, is the following: Let $S$ be a ringed space over $k$, and let us consider couples $(U,f)$ where $U$ is an open subset of $S$ and $f\in\OO_S(U)$ is an angebraic function on $U$. If $h\colon S'\to S$ is a morphism of ringed spaces, then $(h^{-1}(U), h^*(f))$ is a couple on $S'$. Thus, we have a functor 
\[ \aligned F\colon\{\text{Ringed spaces over }k\}&\to \{\text{Sets}\}\\ S&\mapsto F(S)=\{ (U,f): U\in\Open(S), f\in\OO_S(U)\}\endaligned\]

\begin{prop} $F$ is representable by the ringed space over $k$: $$(\PP_1,\OO),$$  where $\OO$ is the sheaf of rings on $\PP_1=\{0<1\}$ defined as follows:
\[ \OO_0=k,\qquad \OO_1=k[x]\] and the morphism $\OO_0\to\OO_1$ is the natural inclusion.
\end{prop}

Now let us consider the functor $F^n=F\times\overset n\cdots\times F$, which assigns to each ringed space $S$ the set of $n$ open subsets of $S$ and $n$ algebraic functions on them. Then:
\begin{cor} $F^n$ is representable by the ringed space
\[ (\PP_n,\OO)\] where $\OO$ is the sheaf of rings on $\PP_n$ defined as follows: for each $\delta=\{i_1,\dots,i_r\}\in\PP_n$, $\OO_\delta=k[x_{i_1},\dots, x_{i_r}]$ and for each $\delta\leq \delta'$, the morphism $\OO_\delta\to\OO_{\delta'}$ is the natural inclusion.
\end{cor}

Finally, let us  consider the following subfunctor $H$ of $F^n$:
\[ H(S)=\left\{ (U_1,f_1),\dots,(U_n,f_n)\in F^n(S):\left\{\aligned &  X=U_1\supseteq U_2\supseteq\dots\supseteq U_n,\\  &  {f_i}_{\vert U_{i+1}} \text{ is invertible, for any }i=1,\dots, n-1\endaligned\right\} \right\}\]

\begin{prop} $H$ is representable by the ringed space
\[ (\Delta_n^\text{\rm top},\OO),\] where $\Delta_n^\text{\rm top}$ is the totally ordered poset $\{1<2<\dots<n\}$ and $\OO$ is the sheaf of rings on it defined by:
\[ \OO_1=k[x_1],\quad \OO_2=k[x_1,x_2]_{x_1},\quad \dots \quad, \quad \OO_n=k[x_1,\dots,x_n]_{x_1\cdot\dots\cdot x_{n-1}}\] and $\OO_i\to\OO_{i+1}$ is the natural inclusion.
\end{prop}

\begin{rem} The ringed spaces $(\PP_n,\OO)$ and $(\Delta_n^\text{\rm top},\OO)$  are not schematic, hence they are not the finite ringed space associated to  a finite affine covering of a scheme.
\end{rem}

\end{document}